 \documentclass[final,1p,times]{elsarticle}

\usepackage{amssymb}
 \usepackage{amsthm}
\usepackage{amsmath,amssymb,amsopn,amsfonts,mathrsfs,amsbsy,amscd}
\usepackage{longtable}
\usepackage{longtable}
\usepackage{caption}
\usepackage{multirow}

\newcommand{\va}{{\varrho}}

\newcommand{\prs}{\langle\;,\;\rangle}

\newcommand{\too}{\longrightarrow}
\newcommand{\om}{\omega}
\newcommand{\esp}{\quad\mbox{and}\quad}

\def\br{[\;,\;]}

\newcommand{\G}{\mathfrak{g}}
\newcommand{\g}{\mathfrak{g}}

\newcommand{\h}{{\mathfrak{h}}}
\newcommand{\ad}{{\mathrm{ad}}}

\newcommand{\tr}{{\mathrm{tr}}}

\newcommand{\B}{{\cal B}}
\newcommand{\M}{{\cal M}}

\newcommand{\D}{{\cal D}}

\newcommand{\p}{{\mathfrak{p} }}

\newcommand{\Om}{\Omega}
\newcommand{\na}{\nabla}

\newcommand{\al}{\alpha}
\newcommand{\be}{\beta}
\newcommand{\ga}{\gamma}
\newcommand{\Ga}{\Gamma}
\newcommand{\e}{\epsilon}

\newcommand{\la}{\lambda}

\newcommand{\de}{\delta}

\newcommand{\inj}{\hookrightarrow}

\newtheorem{defn}{Definition}[section]
\newtheorem{thm}{Theorem}[section]
\newtheorem{prop}{Proposition}[section]

\newtheorem{exem}{Example}

\newtheorem{problem}{Problem}

\font\bb=msbm10

\def\B{\hbox{\bb B}}
\def\R{\hbox{\bb R}}

\begin{document}

\begin{frontmatter}


 

\title{ On Riemann-Poisson  Lie groups}

\author[label1]{Brahim Alioune  }
\address[label1]{Universit\'e de Nouakchott\\
	Facult\'e des sciences et techniques\\
	e-mail: pacha.ali86@gmail.com}
 \author[label2]{Mohamed Boucetta}
 \address[label2]{Universit\'e Cadi-Ayyad\\
 	Facult\'e des sciences et techniques\\
 	BP 549 Marrakech Maroc\\e-mail: m.boucetta@uca.ac.ma}
\author[label3]{Ahmed Sid'Ahmed Lessiad}
\address[label3]{Universit\'e de Nouakchott\\
	Facult\'e des sciences et techniques\\
	e-mail: lessiadahmed@gmail.com}



\begin{abstract}  A Riemann-Poisson Lie group is a Lie group endowed with a  left invariant Riemannian metric and a left invariant Poisson tensor which are compatible in the sense introduced in \cite{bou1}. We study these Lie groups and we give a characterization of their  Lie algebras. We give also a way of building these Lie algebras and we give the list of such Lie algebras up to dimension 5.
\end{abstract}
\end{frontmatter}

\section{Introduction}\label{section1}

In this paper, we study Lie groups endowed with a left invariant Riemannian metric and a left invariant Poisson tensor satisfying a compatibility condition to be defined below. They constitute a subclass of the class of {\it Riemann-Poisson manifolds} introduced and studied by the second author (see \cite{ bou3, bou4, bou1, bou2}).

Let $(M,\pi,\prs)$ be smooth manifold endowed with a Poisson tensor $\pi$ and a Riemannian metric $\prs$. We denote by $\prs^*$ the Euclidean product on $T^*M$ naturally associated to $\prs$. The Poisson tensor defines a Lie algebroid structure on $T^*M$ where the anchor map is the contraction $\#_\pi : T^*M \longrightarrow TM$ 
given by $\prec\be,\#_\pi(\al)\succ=\pi(\al,\be)$ and the Lie bracket on $\Om^1(M)$ is the Koszul bracket given by
\begin{equation} \label{k}[\alpha,\beta]_\pi = \mathcal{L}_{\#_\pi(\alpha)}\beta - \mathcal{L}_{\#_\pi(\beta)}\alpha -d\pi(\alpha,\beta),\quad\al,\be\in\Om^1(M). \end{equation}
This Lie algebroid structure and the metric $\prs^*$  define a contravariant connection $\D:\Om^1(M)\times\Om^1(M)\too \Om^1(M)$  by Koszul formula
\begin{eqnarray}\label{lv} 2\langle\D_\al\be,\ga\rangle^*&=&\#_\pi(\al).\langle\be
	,\ga\rangle^*+\#_\pi(\be).\langle\al ,\ga\rangle^*-\#_\pi(\ga).\langle \al,\be\rangle^*\\ \nonumber
	&&+\langle[\al,\be]_{\pi},\ga \rangle^*+\langle[\ga,\al]_{\pi}, \be\rangle^*+\langle[\ga,\be]_{\pi}, \al\rangle^*,\quad\al,\be,\ga\in \Om^1(M).
\end{eqnarray}
This is the unique  torsionless contravariant connection which is metric, i.e., for any $\al,\be,\ga\in \Om^1(M)$, 
\[ \D_\alpha\beta - \D_\beta\alpha = [\alpha,\beta]_\pi\esp
\#_\pi(\alpha).\langle\beta,\gamma\rangle^* = \langle\D_\alpha\beta,\gamma\rangle^* + \langle\beta,\D_\alpha\gamma\rangle^*.  \]
The notion of contravariant connection was introduced by Vaisman in \cite{vaisman} and studied in more details by Fernandes in the context of Lie algebroids \cite{fern}. The connection $\D$ defined above is called {\it contravariant Levi-Civita connection} associated to the couple $(\pi,\prs)$ and it appeared first in \cite{bou3}.

 The triple $(M,\pi,\prs)$ is called a {\it Riemannian-Poisson manifold} if $\D\pi=0$, i.e., for any $\alpha, \beta, \gamma \in \Omega^1(M),$
	\begin{equation}\label{rp}\D\pi(\alpha,\beta,\gamma) := \#_\pi(\alpha).\pi(\beta,\gamma) - \pi(\D_\alpha\beta,\gamma) + \pi(\beta,\D_\alpha\gamma)=0.\end{equation}
This notion  was introduced by the second author in \cite{bou3}. Riemann-Poisson manifolds turned out to have interesting geometric properties (see\cite{ bou3, bou4, bou1, bou2}).	Let's mention some of them.\begin{enumerate}
	\item 
The condition of compatibility \eqref{rp} is weaker than the condition $\na\pi=0$ where $\na$ is the Levi-Civita connection of $\prs$. Indeed, the condition \eqref{rp}  allows the Poisson tensor to have a variable rank. For instance, linear Poisson structures which are Riemann-Poisson exist and were characterized in \cite{bou2}. Furthermore, 
 let $(M,\prs)$ be a Riemannian manifold and $(X_1,\ldots,X_r)$ a family of commuting Killing vector fields. Put
	\[ \pi=\sum_{i,j}X_i\wedge X_j. \]Then $(M,\pi,\prs)$ is a Riemann-Poisson manifold. This  example illustrates also the weakness of the condition \eqref{rp} and, more importantly, it is the local model of the geometry of noncommutative deformations studied by Hawkins (see  \cite[Theorem 6.6]{Hawkins}).

\item
 Riemann-Poisson manifolds  can be thought of as a generalization of K\"ahler manifolds. Indeed, let $(M,\pi,\prs)$ be a Poisson manifold endowed with a Riemannian metric such that $\pi$ is invertible. Denote by $\om$ the symplectic form inverse of $\pi$. Then $(M,\pi,\prs)$ is Riemann-Poisson manifold if and only if $\na\om=0$ where $\na$ is the Levi-Civita connection of $\prs$. In this case, if we define $A:TM\too TM$ by $\om(u,v)=\langle Au,v\rangle$ then $-A^2$ is symmetric definite positive and hence there exists a unique $Q:TM\too TM$ symmetric definite positive such that $Q^2=-A^2$. It follows that $J=AQ^{-1}$ satisfies $J^2=-\mathrm{Id}_{TM}$, skew-symmetric with respect $\prs$ and $\na J=0$. Hence $(M,J,\prs)$ is a K\"ahler manifold and its K\"ahler form $\om_J(u,v)=\langle Ju,v\rangle$ is related to $\om$ by the following formula: \begin{equation}\label{sy}
\om(u,v)=-\om_J\left(\sqrt{-A^2}u,v\right),\quad u,v\in TM.\end{equation}

Having this construction in mind, we will call in this paper a K\"ahler manifold a triple $(M,\prs,\om)$ where $\prs$ is a Riemannian metric and $\om$ is a nondegenerate 2-form $\om$ such that $\na\om=0$ where $\na$ is the Levi-Civita connection of $\prs$. 

\item The symplectic foliation of a Riemann-Poisson manifold when $\pi$ has a constant rank has an important property namely it is both a Riemannian foliation and a K\"ahler foliation. 

Recall that a Riemannian foliation is a foliated manifold $(M,\mathcal{F})$ with a Riemannian metric $\prs$ such that the orthogonal distribution $T^\perp\mathcal{F}$ is totally geodesic. 

 K\"ahler foliations are a generalization of K\"ahler manifolds (see \cite{deninger}) and, as for the notion of K\"ahler manifold, we call in this paper a K\"ahler foliation a foliated manifold $(M,\mathcal{F})$ endowed with a leafwise metric $\prs_{\mathcal{F}}\in\Ga(\otimes^2T^*\mathcal{F})$ and a nondegenerate leafwise differential 2-form $\om_{\mathcal{F}}\in\Ga(\otimes^2T^*\mathcal{F})$ such  any leaf with the restrictions of $\prs_{\mathcal{F}}$ and $\om_{\mathcal{F}}$ is a K\"ahler manifold.

\begin{thm}[\cite{bou1}]\label{thfo} Let $(M,\prs,\pi)$ be a Riemann-Poisson manifold with $\pi$ of constant rank. Then its symplectic foliation is both a Riemannian and a K\"ahler foliation.
	
\end{thm}
\end{enumerate}

Having in mind these properties particularly Theorem \ref{thfo}, it will be interesting to find large classes of examples of Riemann-Poisson manifolds. This paper will describe the rich collection of examples which are obtained by providing an arbitrary Lie group $G$ with a Riemannian metric $\prs$ and a Poisson tensor $\pi$ invariant under left translations and such that $(G,\prs,\pi)$ is Riemann-Poisson. We call $(G,\prs,\pi)$ a {\it Riemann-Poisson Lie group}.
This class of examples can be enlarged substantially, with no extra work, as follows. If $(G,\prs,\pi)$ is a Riemann-Poisson Lie group and $\Ga$ is any discrete subgroup of $G$ then $\Ga\backslash G$ carries naturally a structure of Riemann-Poisson manifold. 

The paper is organized as follows. In Section \ref{section2}, we give the material needed in the paper and we describe the infinitesimal counterpart of Riemann-Poisson Lie groups, namely, Riemann-Poisson Lie algebras. In Section \ref{section3}, we prove our main result  which gives an useful description of Riemann-Poisson Lie algebras (see Theorem \ref{main1}). We use this theorem in Section \ref{section4} to derive  a	 method for building Riemann-Poisson Lie algebras. We explicit this method by giving the list of Riemann-Poisson Lie algebras up to dimension 5.

\section{Riemann-Poisson Lie groups and their infinitesimal characterization}\label{section2}

Let $G$ be a Lie group and $(\G=T_eG,\br)$ its Lie algebra.
\begin{enumerate}
	\item A left invariant Poisson tensor $\pi$ on $G$ is entirely determined by 
	\[ \pi(\al,\be)(a)=r(\mathrm{L}_{a}^*\al,\mathrm{L}_{a}^*\beta), \]where $a\in G,\al,\be\in T_{a}^*G$, $\mathrm{L}_a$ is the left multiplication by $a$ and $r\in\wedge^2\G$ satisfies the classical Yang-Baxter equation 
	\begin{equation}\label{yb1} [r,r]=0, \end{equation}where $[r,r]\in\wedge^3\G$ is given by
	\begin{equation}\label{yb}
	[r,r](\al,\be,\ga):=\prec\al,[r_\#(\be),r_\#(\ga)]\succ+
	\prec\be,[r_\#(\ga),r_\#(\al)]\succ+\prec\ga,[r_\#(\al),r_\#(\be)]\succ,\quad\al,\be,\ga\in\G^*,
	\end{equation}and $r_\#:\G^*\too\G$ is the contraction associated to $r$. In this case, the Koszul bracket \eqref{k} when restricted to left invariant differential 1-forms induces a Lie bracket on $\G^*$ given by
	\begin{equation}\label{br} [\alpha,\beta]_r = \ad^*_{r_\#(\al)}\be - \ad^*_{r_\#(\be)}\al,\quad\al,\be\in\G^*, \end{equation}where $\prec\ad^*_u\al,v\succ=-\prec\al,[u,v]\succ$. Moreover, $r_\#$ becomes a morphism of Lie algebras, i.e.,
	\begin{equation}\label{mo}
	r_\#([\al,\be]_r)=[r_\#(\al),r_\#(\be)],\quad\al,\be\in\G^*.
	\end{equation}
	\item A let invariant Riemannian metric $\prs$ on $G$ is entirely determined by
	\[ \langle u,v\rangle(a)=\rho(T_a\mathrm{L}_{a^{-1}}u,T_a\mathrm{L}_{a^{-1}}v), \]where $a\in G,u,v\in T_aG$ and $\rho$ is a scalar product on $\G$. The Levi-Civita connection of $\prs$ is left invariant and induces a product $A:\G\times\G\too\G$ given by
	\begin{equation}\label{lp}
	2\va(A_uv,w)=\va([u,v],w)+\va([w,u],v)+\va([w,v],u),\quad u,v,w\in\G.
	\end{equation} 
	It is the unique product on $\G$ satisfying
	\[ A_uv-A_vu=[u,v]\esp \va(A_uv,w)+\va(v,A_uw)=0, \]for any $u,v,w\in\G$. We call $A$ the {\it Levi-Civita product} associated to $(\G,\br,\rho)$.
	
	\item Let $(G,\prs,\Om)$ be a Lie group endowed with a left invariant Riemannian metric and a nondegenerate left invariant 2-form. Then $(G,\prs,\Om)$ is a K\"ahler manifold if and only if, for any $u,v,w\in\G$,
	\begin{equation}\label{ka}
	\om(A_uv,w)+\om(u,A_uv)=0,
	\end{equation}where $\om=\Om(e)$, $\rho=\prs(e)$ and $A$ is the Levi-Civita product of $(\G,\br,\rho)$. In this case we call $(\G,\br,\rho,\om)$ a K\"ahler Lie algebra. 
	
\end{enumerate}

As all the left invariant structures on Lie groups, Riemann-Poison Lie groups can be characterized at the level of their Lie algebras.

\begin{prop}\label{pr1} Let $(G,\pi,\prs)$ be a Lie group endowed with a left invariant bivector field and a left invariant metric and $(\G,\br)$ its Lie algebra. Put $r=\pi(e)\in\wedge^2\G$, $\va=\prs_e$ and $\va^*$ the associated Euclidean product on $\G^*$. Then $(G,\pi,\prs)$ is a Riemann-Poisson Lie group if and only if
	\begin{enumerate}
		\item[$(i)$] $[r,r]=0$,
		\item[$(ii)$] for any $\al,\be,\ga\in\G^*$, $r(A_\al\be,\ga)+r(\be,A_\al\ga)=0$,
	\end{enumerate}where $A$ is the Levi-Civita product associated to $(\G^*,\br_r,\va^*)$.
	
\end{prop}

\begin{proof} For any $u\in\G$ 	and $\al\in\G^*$, we denote by $u^\ell$ and $\al^\ell$, respectively, the left invariant vector field and the left invariant differential 1-form on $G$ given by
	\[ u^\ell(a)=T_e\mathrm{L}_a(u)\esp \al^\ell(a)=T_a^*\mathrm{L}_{a^{-1}}(\al),\quad a\in G,\; \mathrm{L}_a(b)=ab. \]
	Since $\pi$ and $\prs$ are left invariant, one can see easily from \eqref{k} and \eqref{lv} that we have, for any $\al,\be,\ga\in\G^*$,
	\[ \begin{cases} [\pi,\pi]_S(\al^\ell,\be^\ell,\ga^\ell)=[r,r](\al,\be,\ga),\;\#_\pi(\al^\ell)=(r_\#(\al))^\ell,\;\mathcal{L}_{\#_\pi(\al^\ell)}\be^\ell=
	\left(\ad_{r_\#(\al)}^*\be\right)^\ell,\\
	[\al^\ell,\be^\ell]_\pi=\left([\al,\be]_r\right)^\ell,\;\D_{\al^\ell}\be^\ell=\left(A_\al\be\right)^\ell.
	\end{cases} \] 
	The proposition follows from these formulas,  \eqref{rp} and the fact that $(G,\pi,\prs)$ is a Riemann-Poisson Lie group if and only if, for any $\al,\be,\ga\in\G^*$,
	\[ [\pi,\pi]_S(\al^\ell,\be^\ell,\ga^\ell)=0\esp \D\pi(\al^\ell,\be^\ell,\ga^\ell)=0. \qedhere\]
\end{proof}

Conversely, given a triple $(\G,r,\va)$ where $\G$ is a real Lie algebra,  $r\in\wedge^2\G$ and $\va$ a Euclidean product on $\G$ satisfying the conditions $(i)$ and $(ii)$ in Proposition \ref{pr1} then, for any Lie group $G$ whose Lie algebra is $\G$, if $\pi$ and $\prs$ are the left invariant bivector field and the left invariant metric associated to $(r,\va)$ then $(G,\pi,\prs)$ is a Riemann-Poisson Lie group.
\begin{defn}
A {\it Riemann-Poisson Lie algebra} is a triple $(\G,r,\va)$ where $\G$ is a real Lie algebra,  $r\in\wedge^2\G$ and $\va$ a Euclidean product on $\G$ satisfying the conditions $(i)$ and $(ii)$ in Proposition \ref{pr1}.\end{defn}

To end this section, we give another characterization of the solutions of the classical Yang-Baxter equation \eqref{yb1} which will be useful later.

We observe that  $r\in \wedge^2{\G}$ is
equivalent to the data of a vector subspace $S\subset\G$ and a
nondegenerate 2-form $\om_r\in\wedge^2S^*$.

Indeed, for $r\in \wedge^2{\G}$, we put $S=\mathrm{Im}r_\#$ and
$\om_r(u,v)=r(r_\#^{-1}(u),r_\#^{-1}(v))$ where $u,v\in S$ and
$r_\#^{-1}(u)$ is any antecedent of $u$ by $r_\#$.

Conversely, let $(S,\om)$ be  a vector subspace of $\G$
with a non-degenerate 2-form. The 2-form $\om$ defines an
isomorphism $\om^b:S\too S^*$ by $\om^b(u)=\om(u,.)$, we denote by
$\#:S^*\too S$ its inverse and we put $r_\#=\#\circ i^*$ where
$i^*:{\G}^*\too S^*$ is the dual of the inclusion $i:S\inj
\G$.

With this observation in mind, the following proposition gives
another description of the solutions of the Yang-Baxter equation.
\begin{prop}\label{pr2} Let $r\in \wedge^2{\G}$ and $(S,\om_r)$ its
	associated vector subspace. The following assertions are equivalent:
	\begin{enumerate}
		\item $[r,r]=0.$

		\item $S$ is a subalgebra of $\G$ and $$\de\om_r(u,v,w):=\om_r(
		u,[v,w])+\om_r(v ,[w,u])+\om_r( w,[u,v])=0$$for any $u,v,w\in
		S$.
		
	\end{enumerate}

\end{prop}
\begin{proof} The proposition follows from the following  formulas:
	$$\prec\ga,r_\#([\al,\be]_r)-[r_\#(\al),r_\#(\be)]\succ=-[r,r](\al,\be,\ga),
	\qquad\al,\be,\ga\in\G^*$$and, if
	$S$ is a subalgebra,
	$$[r,r](\al,\be,\ga)=-\de\om_r(r_\#(\al),r_\#(\be),r_\#(\ga)).\qedhere$$\end{proof}

This proposition shows that there is a correspondence between the set of solutions of the Yang-Baxter
equation the set of symplectic subalgebras of $\G$. We
recall that a symplectic algebra  is a Lie   algebra $S$
endowed with a non-degenerate 2-form $\om$ such that $\de\om=0.$

\section{A characterization of Riemann-Poisson Lie algebras}\label{section3}

In this section, we combine Propositions \ref{pr1} and \ref{pr2} to establish a characterization of Riemann-Poisson Lie algebras which will be used later to build such  Lie algebras. We establish first an intermediary result.

\begin{prop}\label{pr0} Let $(\G,r,\va)$ be a Lie algebra endowed with $r\in\wedge^2\G$ and a Euclidean product $\va$. Denote by $\mathcal{I}=\ker r_\#$, $\mathcal{I}^\perp$ its orthogonal with respect to $\va^*$ and $A$ the Levi-Civita product associated to $(\G^*,\br_r,\va^*)$. Then $(\G,r,\va)$ is a Riemann-Poisson Lie algebra if and only if:
	
	\begin{enumerate}\item[$(c_1)$] $[r,r]=0$.
		\item[$(c_2)$] For all $\alpha \in \mathcal{I}, A_\alpha = 0.$
		
		\item[$(c_3)$] For all $\alpha, \beta, \gamma \in \mathcal{I}^\perp,$ 
		$A_\al\be\in \mathcal{I}^\perp$ and
		$$
			r(A_\alpha\beta,\gamma) + r(\beta,A_\alpha\gamma) = 0.
		$$
	\end{enumerate}
\end{prop}	
	
\begin{proof} By using the splitting $\G^*=\mathcal{I}\oplus \mathcal{I}^\perp$, on can see that the conditions $(i)$ and $(ii)$ in Proposition \ref{pr1} are equivalent to 
	\begin{equation}\label{eq1} \begin{cases}[r,r]=0,\\ r(A_\alpha\beta,\gamma)  = 0,\al\in \mathcal{I}, \be\in \mathcal{I},\ga\in \mathcal{I}^\perp,\\
	r(A_\alpha\beta,\gamma) + r(\beta,A_\alpha\gamma) = 0,\al\in \mathcal{I}, \be\in \mathcal{I}^\perp,\ga\in \mathcal{I}^\perp,\\
	r(A_\alpha\beta,\gamma)  = 0,\al\in \mathcal{I}^\perp, \be\in \mathcal{I},\ga\in \mathcal{I}^\perp,\\
	r(A_\alpha\beta,\gamma) + r(\beta,A_\alpha\gamma) = 0,\al\in \mathcal{I}^\perp, \be\in \mathcal{I}^\perp,\ga\in \mathcal{I}^\perp.
	\end{cases} \end{equation}
	Suppose that the conditions $(c_1)$-$(c_3)$ hold. Then for any $\al\in \mathcal{I}$ and $\be\in\mathcal{I}^\perp$, $A_\be\al=[\be,\al]_r$ and hence $r_\#(A_\be\al)=[r_\#(\be),r_\#(\al)]=0$ and hence the equations in \eqref{eq1}	holds.

Conversely, suppose that \eqref{eq1} holds. Then $(c_1)$ holds obviously. 

For any $\al,\be\in\mathcal{I}$, the second equation in \eqref{eq1} is equivalent to  $A_\al\be\in\mathcal{I}$ and we have from \eqref{br} and \eqref{lp}
$[\al,\be]_r=0$ and $A_\al\be\in \mathcal{I}^\perp$. Thus $A_\al\be=0$. 

Take now $\al\in \mathcal{I}$ and $\be\in \mathcal{I}^\perp$. For any $\ga\in \mathcal{I}$, $\va^*(A_\al\be,\ga)=-\va^*(\be,A_\al\ga)=0$ and hence $A_\al\be\in \mathcal{I}^\perp$.
On the other hand, 
$$r_\#([\al,\be]_r)=r_\#(A_\al\be)-r_\#(A_\be\al)\stackrel{\eqref{mo}}=[r_\#(\al),r_\#(\be)]=0.$$ So,  for any $\ga\in\mathcal{I}^\perp$,
\begin{eqnarray*}
\prec\ga,r_\#(A_\al\be)\succ&=&\prec\ga,r_\#(A_\be\al)\succ
=r(A_\be\al,\ga)
\stackrel{\eqref{eq1}}=0.
\end{eqnarray*}
This shows that $A_\al\be\in \mathcal{I}$ and hence $A_\al\be=0$. Finally, $(c_2)$ is  true. Now, for any $\al\in \mathcal{I}^\perp$, the fourth equation in \eqref{eq1} implies that $A_\al$ leaves invariant $\mathcal{I}$ and since it is skew-symmetric it leaves invariant $\mathcal{I}^\perp$ and $(c_3)$ follows.   This completes the proof.
\end{proof}

\begin{prop}\label{prbi} Let $(\G,\va,r)$ be a Lie algebra endowed with a solution of classical Yang-Baxter equation and a bi-invariant Euclidean product, i.e.,
	\[ \va(\ad_uv,w)+\va(v,\ad_uw)=0,\quad u,v,w\in\G. \]
Then $(\G,\va,r)$ is Riemann-Poisson Lie algebra if and only if $\mathrm{Im}r_\#$ is an abelian subalgebra. 	
\end{prop}

\begin{proof} Since $\va$ is bi-invariant, one can see easily that for any $u\in\G$, $\ad_u^*$ is skew-symmetric with respect to $\va^*$ and hence the Levi-Civita product $A$ associated to $(\G^*,\br_r,\va^*)$ is given by $A_\al\be=\ad_{r_\#(\al)}^*\be$. So, $(\G,\va,r)$ is Riemann-Poisson Lie algebra if and only if, for any $\al,\be,\ga\in\G^*$,
	\begin{eqnarray*}
	0&=&r(\ad_{r_\#(\al)}^*\be,\ga)+r(\be,\ad_{r_\#(\al)}^*\ga)\\
	&=&\prec\be,[r_\#(\al),r_\#(\ga)]\succ-\prec\ga,[r_\#(\al),r_\#(\be)]\succ\\
	&\stackrel{\eqref{yb1}}=&\prec\al,[r_\#(\be),r_\#(\ga)]\succ
	\end{eqnarray*}and the result follows.
	\end{proof}	
 
 Let $(\G,\br)$ be a Lie algebra, $r\in\wedge^2\G$ and $\va$ a Euclidean product on $\G$. Denote by $(S,\om_r)$ the symplectic vector subspace associated to $r$ and by $\#:\G^*\too\G$ the isomorphism given by $\va$. Note that the Euclidean product on $\G^*$ is given by $\va^*(\al,\be)=\va(\#(\al),\#(\be))$.
 We have
 \[ \G^*=\mathcal{I}\oplus \mathcal{I}^\perp\esp \G=S\oplus S^\perp,\]where $\mathcal{I}=\ker r_\#$.
 Moreover, $r_\#:\mathcal{I}^\perp\too S$ is an isomorphism, we denote by $\tau:S\too \mathcal{I}^\perp$ its inverse. From the relation
 \[ \va(\#(\al),r_\#(\be))=\prec\al,r_\#(\be)\succ=r(\be,\al), \]we deduce that
 $\#:\mathcal{I}\too S^\perp$ is an isomorphism and hence $\#:\mathcal{I}^\perp\too S$ is also an isomorphism.

 Consider the isomorphism $J:S\too S$ linking $\om_r$ to $\va_{|S}$, i.e.,
 \[ \om_r(u,v)=\rho(Ju,v),\quad u,v\in S. \]On can see easily that $J=-\#\circ\tau$.

\begin{thm}\label{main1} With the notations above, $(\G,r,\va)$ is a Riemann-Poisson Lie algebra if and only if the following conditions hold:
	\begin{enumerate}
		\item $(S,\va_{|S},\om_r)$ is a K\"ahler Lie subalgebra, i.e., for all $s_1, s_2, s_3 \in S$,
		\begin{equation}\label{om} \omega_r(\nabla_{s_1}s_2 , s_3) + \omega_r(s_2 , \nabla_{s_1}s_3) = 0,\end{equation}
		where $\nabla$ is the Levi-Civita product associated to $(S ,\br, \va_{|S})$.
		\item for all $s \in S$ and all $u, v \in S^\perp,$
		\begin{equation}\label{perp}
			\va(\phi_{S}(s)(u),v) + \va(u,\phi_{S}(s)(v)) = 0,
		\end{equation}where $\phi_{S}:S\too \mathrm{End}(S^\perp)$, $u\mapsto \mathrm{pr}_{S^\perp}\circ\ad_u$ and $\mathrm{pr}_{S^\perp}:\G\too S^\perp$ is the orthogonal projection.
				\item For all $s_1, s_2 \in S$ and all $u \in S^\perp,$
		\begin{equation}\label{oms}
			\om_r(\phi_{S^\perp}(u)(s_1),s_2)+\om_r(s_1,\phi_{S^\perp}(u)(s_1))=0,
		\end{equation}where $\phi_{S^\perp}:S^\perp\too \mathrm{End}(S)$, $u\mapsto \mathrm{pr}_{S}\circ\ad_u$ and $\mathrm{pr}_{S}:\G\too S$ is the orthogonal projection.
	\end{enumerate}
\end{thm}
\begin{proof} Suppose first that $(\G,r,\va)$ is a Riemann-Poisson Lie algebra.
	 According to Propositions \ref{pr0} and \ref{pr2}, this is equivalent to
	 \begin{equation}\label{c}
	 \begin{cases}
	 (S,\om_r)\;\mbox{ is a symplectic subalgebra},\\
	 \forall \al\in \mathcal{I},\; A_\al=0,\\
	 \forall\; \al,\be,\ga\in \mathcal{I}^\perp,\; 
	 A_\al\be\in \mathcal{I}^\perp\esp r(A_\al\be,\ga)+r(\be,A_\al\ga)=0,
	 \end{cases}
	 \end{equation}where $A$ is the Levi-Civita product of $(\G^*,\br_r,\va^*)$.

For $\al,\be\in \mathcal{I}$ and $\ga\in \mathcal{I}^\perp$,
\begin{eqnarray}
2\va^*(A_\al\be,\ga)&=&\va^*([\al,\be]_r,\ga)+\va^*([\ga,\be]_r,\al)+
\va^*([\ga,\al]_r,\be)\nonumber\\
&=&\va^*(\ad_{r_\#(\ga)}^*\be,\al)	+\va^*(\ad_{r_\#(\ga)}^*\al,\be)\nonumber\\
&=&-\prec \be,[r_\#(\ga),\#(\al)]\succ-\prec \al,[r_\#(\ga),\#(\be)]\succ\nonumber\\
&=&-\va(\#(\be),[r_\#(\ga),\#(\al)])	-\va(\#(\al),[r_\#(\ga),\#(\be)]).\label{i}
\end{eqnarray}	Since $\#:\mathcal{I}\too S^\perp$ and $r_\#: \mathcal{I}^\perp\too S$ are isomorphisms, we deduce from \eqref{i} that $A_\al\be=0$ for any $\al,\be\in \mathcal{I}$ is equivalent to \eqref{perp}.

For $\al\in \mathcal{I}$ and $\be,\ga\in \mathcal{I}^\perp$,
\begin{eqnarray}
	2\va^*(A_\al\be,\ga)&=&\va^*([\al,\be]_r,\ga)+\va^*([\ga,\be]_r,\al)+
	\va^*([\ga,\al]_r,\be)\nonumber\\
	&=&-\va^*(\ad_{r_\#(\be)}^*\al,\ga)-\va^*(\ad_{r_\#(\be)}^*\ga,\al)
	+\va^*(\ad_{r_\#(\ga)}^*\be,\al)	+\va^*(\ad_{r_\#(\ga)}^*\al,\be)\nonumber\\
	&=&\prec \al,[r_\#(\be),\#(\ga)]\succ+\prec \ga,[r_\#(\be),\#(\al)]\succ-\prec \be,[r_\#(\ga),\#(\al)]\succ-\prec \al,[r_\#(\ga),\#(\be)]\succ\nonumber\\
	&=&\va(\#(\ga),[r_\#(\be),\#(\al)])-\va(\#(\be),[r_\#(\ga),\#(\al)])+\prec \al,[r_\#(\be),\#(\ga)]\succ-\prec \al,[r_\#(\ga),\#(\be)]\succ\nonumber\\
	&=&-\va(J\circ r_\#(\ga),[r_\#(\be),\#(\al)])+\va(J\circ r_\#(\be),[r_\#(\ga),\#(\al)])+\prec \al,[r_\#(\be),\#(\ga)]\succ-\prec \al,[r_\#(\ga),\#(\be)]\succ\nonumber\\
	&=&-\om_r( r_\#(\ga),\mathrm{pr}_S([r_\#(\be),\#(\al)]))-\om_r(\mathrm{pr}_S([r_\#(\ga),\#(\al)]), r_\#(\be))\nonumber\\
	&&+\prec \al,[r_\#(\be),\#(\ga)]\succ-\prec \al,[r_\#(\ga),\#(\be)]\succ.\label{ii}
\end{eqnarray} Now, $\#(\be),\#(\ga)\in S$ and $r_\#(\be),r_\#(\ga)\in S$ and since $S$ is a subalgebra we deduce that $[r_\#(\be),\#(\ga)],[r_\#(\ga),\#(\be)]\in S$ and hence
\[ \prec \al,[r_\#(\be),\#(\ga)]\succ=\prec \al,[r_\#(\ga),\#(\be)]\succ=0. \]
 We have also $\#:\mathcal{I}\too S^\perp$ and $r_\#: \mathcal{I}^\perp\too S$ are isomorphisms so that, by virtue of \eqref{ii}, $A_\al\be=0$ for any $\al\in \mathcal{I}$ and $\be\in \mathcal{I}^\perp$ is equivalent to \eqref{oms}.

On the other hand,
for any $\al,\be,\ga\in \mathcal{I}^\perp$, since $\#=-J\circ r_\#$, the relation
\[ 2\va^*(A_\al\be,\ga)=\va^*([\al,\be]_r,\ga)+\va^*([\ga,\be]_r,\al)+\va^*([\ga,\al]_r,\be) \]can be written
\[ 2\va(J\circ r_\#(A_\al\be),J\circ r_\#(\ga))=
\va(J\circ r_\#([\al,\be]_r),J\circ r_\#(\ga))+\va(J\circ r_\#([\ga,\be]_r),J\circ r_\#(\al))+\va(J\circ r_\#([\ga,\al]_r),J\circ r_\#(\be)). \]But $r_\#([\al,\be]_r)=[r_\#(\al),r_\#(\be)]$ and hence
\[ 2\langle r_\#(A_\al,\be),r_\#(\ga))\rangle_J=\langle [r_\#(\al),r_\#(\be)],r_\#(\ga)\rangle_J+\langle [r_\#(\g	),r_\#(\be)],r_\#(\al)\rangle_J+\langle [r_\#(\ga),r_\#(\al)],r_\#(\be)\rangle_J, \]where $\langle u,v\rangle_J=\va(Ju,Jv)$. This shows that $r_\#(A_\al\be)=\na_{r_\#(\al)}r_\#(\be)$ where $\na$ is the Levi-Civita product of $(S,\br,\prs_J)$ and the third relation in \eqref{c} is equivalent to
\[ \om_r(\na_uv,w)+\om_r(v,\na_uw)=0,\quad u,v,w\in S. \]
This is equivalent to $\na_uJv=J\na_uv$. Let us show that $\na$ is actually the Levi-Civita product of $(S,\br,\va)$. Indeed, for any $u,v,w\in S$, $\na_uv-\na_vu=[u,v]$ and
\begin{eqnarray*} \va(\na_uv,w)+\va(\na_uw,v)&=&\langle J^{-1}\na_uv,J^{-1}w\rangle_J +\langle J^{-1}\na_uw,J^{-1}v\rangle_J \\
&=&\langle \na_uJ^{-1}v,J^{-1}w\rangle_J +\langle \na_uJ^{-1}w,J^{-1}v\rangle_J\\
&=&0.	
	\end{eqnarray*}
So we have shown the direct part of the theorem. The converse can be deduced easily from the relations we established in the proof of the direct part.
	\end{proof}

\begin{exem}\begin{enumerate}
		\item Let $G$ be a compact connected Lie group, $\G$ its Lie algebra and $T$ an even dimensional torus of $G$. Choose a bi-invariant Riemannian metric $\prs$ on $G$,  a nondegenerate $\om\in\wedge^2 S^*$ where $S$ is the Lie algebra of $T$ and put $\va=\prs(e)$. Let $r\in\wedge^2\G$ be  the solution  of the classical Yang-Baxter associated to $(S,\om)$. By using either Proposition \ref{prbi} or Theorem \ref{main1}, one can see easily that $(\G,\va,r)$ is a Riemann-Poisson Lie algebra and hence $(G,\prs,\pi)$ is a Riemann-Poisson Lie group where $\pi$ is the left invariant Poisson tensor associated to $r$. According to Theorem \ref{thfo}, the orbits of the right action  of $T$ on $G$ defines a Riemannian and K\"ahler foliation. For instance, $G=\mathrm{SO}(2n)$,  $T=\operatorname{Diagonal}(D_1,\ldots,D_n)$ where $D_i=\left(\begin{matrix}
		\cos(\theta_i)&\sin(\theta_i)\\-\sin(\theta_i)&\cos(\theta_i)
		\end{matrix}\right)$ and $\prs=-K$ where $K$ is the Killing form.

\end{enumerate}
	
\end{exem}

\section{Construction of Riemann-Poisson Lie algebras}\label{section4}

In this section, we give a general method for building Riemann-Poisson Lie algebras and we use  it to give all Riemann-Poisson Lie algebras up to dimension 5.

According to Theorem \ref{main1}, to build Riemann-Poisson Lie algebras one needs to solve the following problem.
\begin{problem} We look for:
	\begin{enumerate}
		\item A K\"ahler Lie algebra $(\h,\br_\h,\va_\h,\om)$,
		\item a  Euclidean vector space $(\p,\va_\p)$,
		\item a bilinear skew-symmetric map $\br_\p:\p\times\p\too\p$,
		\item a bilinear skew-symmetric map ${\mu}:\p\times\p\too\h$,
		\item  two linear  maps $\phi_\p:\p\too\mathrm{sp}(\h,\om)$ and 	$\phi_\h:\h\too \mathrm{so}(\p)$ where $\mathrm{sp}(\h,\om)=\left\{J:\h\too\h, J^\om+J=0 \right\}$ and $\mathrm{so}(\p)=\left\{ A:\p\too\p, A^*+A=0  \right\}$,  $J^\om$ is the adjoint with respect to $\om$ and $A^*$ is the adjoint with respect to $\va_\p$,
	\end{enumerate}
	 such that the bracket $\br$ on $\g=\h\oplus\p$ given, for any $a,b\in\p$ and $u,v\in\h$, by
	\begin{equation}\label{bra} [u,v]=[u,v]_\h,\;[a,b]={\mu}(a,b)+[a,b]_\p,\;[a,u]=-[u,a]=\phi_\p(a)(u)-\phi_\h(u)(a) \end{equation}is a Lie bracket. 
	
	  In this case, $(\g,\br)$ endowed with $r\in\wedge^2\g$ associated to $(\h,\om)$ and the Euclidean product $\va=\va_\h\oplus\va_\p$ becomes, by virtue of Theorem \ref{main1}, a Riemann-Poisson Lie algebra.
	\label{pro1}
\end{problem}

\begin{prop} \label{pr10} With the data and notations of Problem \ref{pro1},  the bracket given by \eqref{bra} is a Lie bracket if and only if, for any $u,v\in\h$ and $a,b,c\in\p$,
\begin{equation}\label{eqpro}	\begin{cases}
	\phi_\p(a)([u,v]_\h)=[u,\phi_\p(a)(v)]_\h+[\phi_\p(a)(u),v]_\h+\phi_\p(\phi_\h(v)(a))(u)-\phi_\p(\phi_\h(u)(a))(v),\\
	\phi_\h(u)([a,b]_\p)=[a,\phi_\h(u)(b)]_\p+[\phi_\h(u)(a),b]_\p+\phi_\h(\phi_\p(b)(u))(a)-\phi_\h(\phi_\p(a)(u))(b),\\
	\phi_\h([u,v]_\h)=[\phi_\h(u),\phi_\h(v)],\\
	\phi_\p([a,b]_\p)(u)=[\phi_\p(a),\phi_\p(b)](u)+[u,{\mu}(a,b)]_\h-{\mu}(a,\phi_\h(u)(b))-{\mu}(\phi_\h(u)(a),b),\\
	\oint [a,[b,c]_\p]_\p=\oint\phi_\h({\mu}(b,c))(a),\\
	\oint\phi_\p(a)({\mu}(b,c))=\oint{\mu}([b,c]_\p,a),
	\end{cases}\end{equation}where $\oint$ stands for the circular permutation.
	
\end{prop}

\begin{proof} The equations follow from  the Jacobi identity applied to  $(a,u,v)$, $(a,b,u)$ and $(a,b,c)$.  
	\end{proof}
	
	We tackle now the task of determining the list of all Riemann-Poisson Lie algebras up to dimension 5. For this purpose, we need
	to solve Problem \ref{pro1} in the following four cases: $(a)$ $\dim\p=1$, $(b)$ $\dim\h=2$ and $\h$ non abelian, $(c)$ $\dim\h=\dim\p=2$ and $\h$ abelian,  $(d)$ $\dim\h=2$, $\dim\p=3$ and $\h$ abelian.

 It is easy to find the solutions of Problem \ref{pro1} when $\dim\p=1$ since in this case $\mathrm{so}(\p)=0$ and the three last equations in \eqref{eqpro} hold obviously. 
\begin{prop}\label{p1} If  $\dim\p=1$ then the solutions of Problem \ref{pro1} are a K\"ahler Lie algebra $(\h,\va,\om)$, $\phi_\h=0$, $\br_\p=0$, $\mu=0$ and $\phi_\p(a)\in\mathrm{sp}(\h,\om)\cap \mathrm{Der}(\h)$ where $a$ is a generator of $\p$ and $\mathrm{Der}(\h)$ the Lie algebra of derivations of $\h$.
	
\end{prop}

Let us solve Problem \ref{pro1} when $\h$ is 2-dimensional non abelian.

		\begin{prop}\label{nab} Let $((\h,\om,\va_\h),(\p,\br_\p,\va_\p),\mu,\phi_\h,\phi_\p)$ be a solution of Problem \ref{pro1} with $\h$ is 2-dimensional non abelian. Then there exists an orthonormal basis $\B=(e_1,e_2)$ of $\h$, $b_0\in\p$ and two constants $\al\not=0$ and $\be\not=0$ such that:
			\begin{enumerate}
				\item[$(i)$] $[e_1,e_2]_\h=\al e_1$, $\om=\be e_1^*\wedge e_2^*$,
				\item[$(ii)$] $(\p,\br_\p,\va_\p)$ is a Euclidean Lie algebra,
				\item[$(iii)$] $\phi_\h(e_1)=0$,  $\phi_\h(e_2)\in\mathrm{Der}(\p)\cap\mathrm{so}(\p)$ and, for any $a\in\p$, $M(\phi_\p(a),\B)= \left(\begin{matrix}
				0&\va_\p(a,b_0)\\0&0
				\end{matrix}  \right)$,
				\item[$(iv)$] for any $a,b\in\p$, $\mu(a,b)=\mu_0(a,b)e_1$ with $\mu_0$ is a 2-cocycle of $(\p,\br_\p)$ satisfying
				\begin{equation}\label{eqh} \mu_0(a,\phi_\h(e_2)b)+\mu_0(\phi_\h(e_2)a,b)=-\va_\p([a,b]_\p,b_0)-\al\mu_0(a,b). \end{equation}
				\end{enumerate}

		\end{prop}\begin{proof}	 Note first that from the third relation in \eqref{eqpro} we get that $\phi_\h(\h)$ is a solvable subalgebra of $\mathrm{so}(\p)$ and hence must be abelian. Since $\h$ is 2-dimensional non abelian then $\dim\phi_\h(\h)=1$ and $[\h,\h]\subset\ker\phi_\h$. So there exists an orthonormal basis $(e_1,e_2)$ of $\h$ such that $[e_1,e_2]_\h=\al e_1$,  
		 $\phi_\h(e_1)=0$ and $\om=\be e_1^*\wedge e_2^*$. If we identify the endomorphisms of $\h$ with their matrices in the basis $(e_1,e_2)$, we get that $\mathrm{sp}(\h,\om)=\mathrm{sl}(2,\R)$ and there exists $a_0,b_0,c_0\in\p$ such that, for any $a\in\p$,
	\[ \phi_\p(a)=\left(\begin{matrix}
	\va_\p(a_0,a)&\va_\p(b_0,a)\\\va_\p(c_0,a)&
	-\va_\p(a_0,a)
	\end{matrix} \right). \]The first equation in  \eqref{eqpro} is equivalent to
	\begin{eqnarray*}
	\al\left( \va_\p(a_0,a) e_1+\va_\p(c_0,a) e_2 \right)&=&
	-\al\va_\p(a_0,a) e_1+\al\va_\p(a_0,a) e_1+
	\va_\p( a_0,\phi_\h(e_2)(a)) e_1+\va_\p(c_0,\phi_\h(e_2)(a)) e_2,
	\end{eqnarray*} for any $a\in\p$. Since $\phi_\h(e_2)$ is sekw-symmetric, this is equivalent to
	\[ \phi_\h(e_2)(a_0)=-\al a_0\esp \phi_\h(e_2)(c_0)=-\al c_0.  \]
	This implies that $a_0=c_0=0$. The second equation in \eqref{eqpro} implies that $\phi_\h(e_2)$ is a derivation of $\br_\p$. If we take $u=e_1$ in the forth equation in \eqref{eqpro}, we get that $[e_1,\mu(a,b)]=0$, for any $a,b\in\p$ and hence $\mu(a,b)=\mu_0(a,b)e_1$. If we take $u=e_2$ in the forth equation in \eqref{eqpro} we get \eqref{eqh}. The  two last equations are equivalent to $\br_\p$ is a Lie bracket and $\mu_0$ is 2-cocycle of $(\p,\br_\p)$.
	\end{proof}
		The following proposition gives the solutions of Problem \ref{pro1} when $\h$ is 2-dimensional  abelian and $\dim\p=2$.
	\begin{prop}\label{ab} Let $((\h,\om,\va_\h),(\p,\br_\p,\va_\p),\mu,\phi_\h,\phi_\p)$ be a solution of Problem \ref{pro1} with $\h$ is 2-dimensional  abelian and $\dim\p=2$. Then one of the following situations occurs:

		\begin{enumerate}
			\item $\phi_\h=0$, $(\p,\br_\p,\va_\p)$ is a 2-dimensional Euclidean Lie algebra, there exists $a_0\in\p$ and $D\in\mathrm{sp}(\h,\om)$ such that,  for any $a\in\p$, $\phi_\p(a)=\va_\p(a_0,a) D$  and there is no restriction on $\mu$. Moreover, $a_0\in[\p,\p]_\p^\perp$ if $D\not=0$.
			\item $\phi_\h=0$, $(\p,\br_\p,\va_\p)$ is a 2-dimensional non abelian Euclidean Lie algebra, $\phi_\p$ identifies $\p$ to a two dimensional subalgebra of $\mathrm{sp}(\h,\om)$ and there is no restriction on $\mu$.
			\item $(\p,\br_\p,\va_\p)$ is a Euclidean abelian Lie algebra and there exists an orthonormal basis $\B=(e_1,e_2)$ of $\h$ and $b_0\in\p$ such that $\om=\al e_1^*\wedge e_2^*$, $\phi_\h(e_1)=0$, $\phi_\h(e_2)\not=0$  and, for any $a\in\p$, $M(\phi_\p(a),\B)=\left(\begin{matrix}
			0&\va_\p( b_0,a)\\0&0
			\end{matrix}  \right)$ and there is no restriction on $\mu$.
		\end{enumerate}

		\end{prop} 
		\begin{proof} Note first that since $\dim\p=2$ the last two equations in \eqref{eqpro} hold obviously and $(\p,\br_\p)$ is a Lie algebra.
			We distinguish two cases:
			\begin{enumerate}
				\item[$(i)$] $\phi_\h=0$. Then \eqref{eqpro} is equivalent to 
				$\phi_\p$ is a representation of $\p$ in $\mathrm{sp}(\h,\om)\simeq\mathrm{sl}(2,\R)$. Since $\mathrm{sl}(2,\R)$ doesn't contain any abelian two dimensional subalgebra, if $\p$ is an abelian Lie algebra   then $\dim\phi_\p(\p)\leq1$ and the first situation occurs. If $\p$ is not abelian then the first or the second situation occurs depending on $\dim\phi_\p(\p)$.
				\item[$(ii)$] $\phi_\h\not=0$.   Since $\dim\mathrm{so}(\p)=1$ there exists an orthonormal basis $\B=(e_1,e_2)$ of $\h$ such that $\phi_\h(e_1)=0$ and $\phi_\h(e_2)\not=0$. We have $\mathrm{sp}(\h,\om)=\mathrm{sl}(2,\R)$ and hence, for any $a\in\p$,
				$M(\phi_\p(a),\B)=\left(\begin{matrix}
				\va_\p (a_0,a)&\va_\p (b_0,a)\\\va_\p (c_0,a)&-\va_\p (a_0,a)
				\end{matrix}  \right)$. Choose an orthonormal basis $(a_1,a_2)$ of $\p$. Then there exists $\la\not=0$ such that $\phi_\h(e_2)(a_1)=\la a_2$ and $\phi_\h(e_2)(a_2)=-\la a_1$.

				The first equation in \eqref{eqpro} is equivalent to 
				\[ \phi_\p(\phi_\h(e_2)(a))(e_1)=0,\quad a\in\p. \]This is equivalent to
				\[ \phi_\p(a_1)(e_1) =\phi_\p(a_2)(e_1)=0.\]
				Then $a_0=c_0=0$ and hence $\phi_\p(a)=\left(\begin{matrix}
				0&\va_\p (b_0,a)\\0&0
				\end{matrix}  \right)$. The second equation in \eqref{eqpro} gives
				\[ \phi_\h(e_2)([a_1,a_2]_\p)=[a_1,\phi_\h(e_2)(a_2)]_\p+[\phi_\h(e_2)(a_1),a_2]_\p+\phi_\h(\phi_\p(a_2)(e_2))(a_2)-\phi_\h(\phi_\p(a_1)(e_2))(a_2), \]and hence $\phi_\h(e_2)([a_1,a_2]_\p)=0$. Thus $[a_1,a_2]_\p=0$. All the other equations in \eqref{eqpro} hold obviously.\qedhere
				
			\end{enumerate}

			\end{proof}
			
			To tackle the last case, we need the determination of 2-dimensional subalgebras of $\mathrm{sl}(2,\R)$.
			\begin{prop}\label{sl} The 2-dimensional subalgebras of $\mathrm{sl}(2,\R)$ are
				\[ \G_1=\left\{ \left(\begin{matrix}
				\al&\be\\0&-\al
				\end{matrix}  \right),\al,\be\in\R   \right\},\G_2=\left\{ \left(\begin{matrix}
				\al&0\\\be&-\al
				\end{matrix}  \right),\al,\be\in\R     \right\},\G_{x}=\left\{ \left(\begin{matrix}
				\al&\frac{2\be-\al}x\\(\al+2\be)x&-\al
				\end{matrix}  \right),\al,\be\in\R   \right\} \]where $x\in\R\setminus\{0\}$. Moreover,  $\G_x=\G_y$ if and only if $x=y$.
				
			\end{prop}
			
			\begin{proof} Let $\g$ be a 2-dimensional subalgebra of $\mathrm{sl}(2,\R)$.   We consider the basis $\B=(h,e,f)$ of $\mathrm{sl}(2,\R)$ given by
				\[ e=\begin{pmatrix}0&1\\0&0\end{pmatrix},\ 
				f=\begin{pmatrix}
				0&0\\1&0
				\end{pmatrix}
				\ {\rm and}\ \ h=\begin{pmatrix}
				1&0\\0&-1
				\end{pmatrix}.\]Then
				\[ [h,e]=2e,\;[h,f]=-2f\esp [e,f]=h. \]
				If $h\in\g$ then $\ad_h$ leaves $\g$ invariant. But $\ad_h$ has three eigenvalues $(0,2,-2)$ with the associated eigenvectors $(h,e,f)$ and hence it restriction to $\g$ has $(0,2)$ or $(0,-2)$ as eigenvalues. Thus $\g=\g_1$ or $\g=\g_2$.
				
				Suppose now that $h\notin\g$. By using the fact that $\mathrm{sl}(2,\R)$ is unimodular, i.e., for any $w\in\mathrm{sl}(2,\R)$ $\tr(\ad_w)=0$,  we can choose a basis $(u,v)$ of $\g$ such that $(u,v,h)$ is a basis of $\mathrm{sl}(2,\R)$ and 
				\[ [u,v]=u,\;[h,u]=au+v\esp [h,v]=du-av-h. \]
				If $(x_1,x_2,x_3)$ and $(y_1,y_2,y_3)$ are the coordinates of $u$ and $v$ in $\B$, the brackets above gives
				\[ \begin{cases}
				-2(x_1y_3-x_3y_1)-x_1=0,\\ 2(x_2y_3-x_3y_2)-x_2=0,\\ x_1y_2-x_2y_1-x_3=0,\end{cases}\begin{cases}y_1=(2-a)x_1,\\ y_2=-(a+2)x_2,\\ y_3=-ax_3,\end{cases}\esp \begin{cases}dx_1= (a+2)y_1,\\ dx_2=(a-2)y_2,\\ dx_3=ay_3+1 .
				\end{cases} \]Note first that if $x_1=0$ then $(x_2,x_3)=(0,0)$ which impossible so we must have $x_1\not=0$ and hence $d=4-a^2$. If we replace in the third equation in the second system and the last equation, we get $x_3=\frac14$ and $y_3=-\frac{a}4$. The third equation in the first system gives $x_2=-\frac{1}{16x_1}$ and hence $y_1=(2-a)x_1$ and $y_2=\frac{(a+2)}{16x_1}$. Thus
				\[ \g=\mathrm{span}\left\{ \left( \begin{matrix}
				\frac14&-\frac{1}{16x_1}\\x_1&-\frac14
				\end{matrix}   \right), \left( \begin{matrix}
				-\frac{a}4&\frac{(a+2)}{16x_1}\\(2-a)x_1&\frac{a}4
				\end{matrix}   \right)   \right\}=\mathrm{span}\left\{ \left( \begin{matrix}
				1&-\frac{1}{x}\\x&-1
				\end{matrix}   \right), \left( \begin{matrix}
				-{a}&\frac{(a+2)}{x}\\(2-a)x&{a}
				\end{matrix}   \right)   \right\};\quad x=4x_1. \]But
				\[ \left( \begin{matrix}
				0&\frac{2}{x}\\2x&0
				\end{matrix}   \right)=a\left( \begin{matrix}
				1&-\frac{1}{x}\\x&-1
				\end{matrix}   \right)+\left( \begin{matrix}
				-{a}&\frac{(a+2)}{x}\\(2-a)x&{a}
				\end{matrix}   \right) \]and hence
				\[ \g=\mathrm{span}\left\{ \left( \begin{matrix}
				1&-\frac{1}{x}\\x&-1
				\end{matrix}   \right), \left( \begin{matrix}
				0&\frac{2}{x}\\2x&0
				\end{matrix}   \right)   \right\}=\g_x. \]One can check easily that $\g_x=\g_y$ if and only if $x=y$. This completes the proof.
			\end{proof}

			The following two propositions give the solutions of Problem \ref{pro1} when $\h$ is 2-dimensional  abelian and $\dim\p=3$.
			\begin{prop}\label{ab3} Let $((\h,\om,\va_\h),(\p,\br_\p,\va_\p),\mu,\phi_\h,\phi_\p)$ be a solution of Problem \ref{pro1} with $\h$ is 2-dimensional  abelian and $\dim\p=3$ and $\phi_\h=0$. Then one of the following situations occurs:
				\begin{enumerate}
					\item[$(i)$] $(\p,\br_\p,\va_\p)$ is 3-dimensional Euclidean Lie algebra, $\phi_\p=0$ and $\mu$ is 2-cocycle for the trivial representation.
					\item[$(ii)$]  $\phi_\p$ is an isomorphism of Lie algebras between $(\p,\br_\p)$ and $\mathrm{sl}(2,\R)$ and there exists an endomorphism $L:\p\too\h$ such that for any $a,b\in\p$,
					\[ \mu(a,b)=\phi_\p(a)(L(b))-\phi_\p(b)(L(a))-L([a,b]_\p). \]
					\item[$(iii)$] There exists a basis $\B_\p=(a_1,a_2,a_3)$ of $\p$, $\al\not=0$, $\be\not=0$, $\ga,\tau\in\R$ such that $\br_\p$ has one of the two following forms
					\[ \begin{cases}
					[a_1,a_2]_\p=0,\;[a_1,a_3]_\p=\be a_1,\\ [a_2,a_3]_\p=\ga a_1+\al a_2,\;\al\not=0,\be\not=0\\
					M(\va_\p,\B_\p)=\mathrm{I}_3
					\end{cases}\quad\mbox{or}\quad
					\begin{cases}
					[a_1,a_2]_\p=[a_1,a_3]_\p=0,\\ [a_2,a_3]_\p=\al a_2,\;\al\not=0,\\
					M(\va_\p,\B_\p)=\left( \begin{matrix}1&\tau&0\\\tau&1&0\\0&0&1
					\end{matrix}   \right).
					\end{cases}
					 \]
					In both cases, there exists an orthonormal basis $\B_\h=(e_1,e_2)$ of $\h$, $x\not=0$, $u\not=0$ and $v\in\R$ such that $\phi_\p$ has one of the following forms
					\[ \begin{cases}
					M(\phi_\p(a_2),\B_\h)=\left(\begin{matrix}
					0&u\\0&0
					\end{matrix}  \right), \\\M(\phi_\p(a_3),\B_\h)= \left(\begin{matrix}
					-\frac{\al}{2}&v\\0&\frac{\al}{2}
					\end{matrix}  \right),\\ \phi_\p(a_1)=0,
					\end{cases}\;\begin{cases}
					M(\phi_\p(a_2),\B_\h)=\left(\begin{matrix}
					0&0\\u&0
					\end{matrix}  \right),\\ \;M(\phi_\p(a_3),\B_\h)= \left(\begin{matrix}
					\frac{\al}{2}&0\\ v&-\frac{\al}{2}
					\end{matrix}  \right),\\\phi_\p(a_1)=0,
					\end{cases}\;\mbox{or}\;\begin{cases}
					M(\phi_\p(a_2),\B_\h)=\left(\begin{matrix}
					u&-\frac{u}x\\ux&-u
					\end{matrix}  \right),\\\; M(\phi_\p(a_3),\B_\h)= \left(\begin{matrix}
					v&-\frac{2v+\al}{2x}\\\frac{2v-\al}{2}x&-v
					\end{matrix}  \right),\\\phi_\p(a_1)=0.
					\end{cases} \]
					Moreover, $\mu$ is a 2-cocycle for $(\p,\br_\p,\phi_\p)$.
				\item[$(iv)$] There exists an orthonormal basis $\B=(a_1,a_2,a_3)$ of $\p$ such that $\phi_\p(a_1)=\phi_\p(a_2)=0$, $\phi_\p(a_3)$ is a non zero element of $\mathrm{sp}(\h,\om)$ and
				\[ \begin{cases}
				[a_1,a_2]_\p=0,\;[a_1,a_3]_\p=\be a_1+\rho a_2,\\ [a_2,a_3]_\p=\ga a_1+\al a_2,\\
				\end{cases}\quad\mbox{or}\quad
				\begin{cases}
				[a_1,a_2]_\p=\al a_2,\;[a_1,a_3]_\p=\rho a_2,\\ [a_2,a_3]_\p=\ga a_2,\al\not=0.
				\end{cases}
				\]Moreover, $\mu$ is a 2-cocycle for $(\p,\br_\p,\phi_\p)$.

				\end{enumerate}
				
			\end{prop}
			
			\begin{proof}
					 In this case, \eqref{eqpro} is equivalent to $(\p,\br_\p)$ is a Lie algebra and $\phi_\p$ is a representation and $\mu$ is a 2-cocycle of $(\p,\br_\p,\phi_\p)$.
				
				 We distinguish four cases:
			\begin{enumerate}
				\item $\phi_\p=0$ and the case $(i)$ occurs.

			\item $\dim\phi_\p(\p)=3$ and hence $\p$ is isomorphic to $\mathrm{sp}(\h,\om)\simeq\mathrm{sl}(2,\R)$ and hence $\mu$ is a coboundary. Thus $(ii)$ occurs.

			\item 	$\dim\phi_\p(\p)=2$ then $\ker\phi_\p$ is a one dimensional ideal of $\p$. But $\phi_\p(\p)$ is a 2-dimensional subalgebra of $\mathrm{sp}(\h,\om)\simeq\mathrm{sl}(2,\R)$, therefore it is non abelian
			so $\p/\ker\p$ is non abelian.
			
			If $\ker\p\subset[\p,\p]_\p$ then $\dim[\p,\p]_\p=2$ so there exists an orthonormal basis $(a_1,a_2,a_3)$ of $\p$ such that $a_1\in\ker\p$ and
			\[ [a_1,a_2]_\p=\xi a_1,\;[a_1,a_3]_\p=\be a_1\esp [a_2,a_3]_\p=\ga a_1+\al a_2,\;\al\not=0,\be\not=0 \]and we must have $\xi=0$ in order to have the Jacobi identity.

			 If $\ker\p\nsubset[\p,\p]$ then $\ker\p\subset Z(\p)$ and $\dim[\p,\p]=1$.
			Then there exits a basis $(a_1,a_2,a_3)$ of $\p$ such that $a_1\in\ker\p$, $a_2\in [\p,\p]$, $a_3\in\{a_1,a_2\}^\perp$ and 
			\[ [a_2,a_3]_\p=\al a_2,\;[a_3,a_1]_\p= [a_1,a_2]_\p=0,\;\al\not=0.\]The matrix of $\va_\p$ in $(a_1,a_2,a_3)$ is given by
			\[ \left( \begin{matrix}1&\tau&0\\\tau&1&0\\0&0&1
			\end{matrix}   \right). \]

			We choose an orthonormal basis $(e_1,e_2)$ of $\h$ and identify $\mathrm{sp}(\h,\om)$ to $\mathrm{sl}(2,\R)$.
			 Now  $\phi_\p(\p)=\{\phi_p(a_2),\phi_p(a_3)  \}$ is a subalgebra of $\mathrm{sl}(2,\R)$ and, according to Proposition \ref{sl},  $\phi_\p(\p)=\G_1$,  $\G_2$ or $\G_x$. But
			 \[ [\G_1,\G_1]=\R e, [\G_2,\G_2]=\R f\esp [\G_x,\G_x]=\left\{ \left(\begin{matrix}
			 u&-\frac{u}x\\ux&-u
			 \end{matrix}  \right)\right\}.\] So in order for $\phi_\p$ to be a representation we must have
			 \[ \phi_\p(a_2)=\left(\begin{matrix}
			 0&u\\0&0
			 \end{matrix}  \right), \esp\phi_\p(a_3)= \left(\begin{matrix}
			 -\frac{\al}{2}&v\\0&\frac{\al}{2}
			 \end{matrix}  \right) \esp\phi_\p(a_1)=0,\]
			 \[ \phi_\p(a_2)=\left(\begin{matrix}
			 0&0\\u&0
			 \end{matrix}  \right), \;\phi_\p(a_3)= \left(\begin{matrix}
			 \frac{\al}{2}&0\\ v&-\frac{\al}{2}
			 \end{matrix}  \right)\esp\phi_\p(a_1)=0,\]
			 or

			\[ \phi_\p(a_2)=\left(\begin{matrix}
			u&-\frac{u}x\\ux&-u
			\end{matrix}  \right),\; \phi_\p(a_3)= \left(\begin{matrix}
			p&-\frac{2p+\al}{2x}\\\frac{2p-\al}{2}x&-p
			\end{matrix}  \right)\esp\phi_\p(a_1)=0.\]
			
			\item 	$\dim\phi_\p(\p)=1$ then $\ker\phi_\p$ is a two dimensional ideal of $\p$. Then there exists an orthonormal basis $(a_1,a_2,a_3)$ of $\p$ such that
			\[ [a_1,a_2]_\p=\al a_2,\;[a_3,a_1]_\p=p a_1+qa_2 \esp [a_3,a_2]_\p=r a_1+sa_2.\]
			The Jacobi identity gives $\al=0$ or $(p,r)=(0,0)$.
			We take $\phi_\p(a_1)=\phi_\p(a_2)=0$ and $\phi_\p(a_3)\in \mathrm{sl}(2,\R)$.\qedhere
		\end{enumerate}	
				\end{proof}		
				
		\begin{prop}\label{ab3b} Let $((\h,\om,\va_\h),(\p,\br_\p,\va_\p),\mu,\phi_\h,\phi_\p)$ be a solution of Problem \ref{pro1} with $\h$ is 2-dimensional  abelian,  $\dim\p=3$ and $\phi_\h\not=0$. Then there exists an orthonormal basis $(e_1,e_2)$ of $\h$, an orthonormal basis $(a_1,a_2,a_3)$ of $\p$, $\la>0$, $\al,p,q,\mu_1,\mu_2,\mu_3\in\R$ such that
			$$\phi_\h(e_1)=0,\;\phi_\h(e_2)(a_1)=\la a_2,\; \phi_\h(e_2)(a_2)=-\la a_1\esp \phi_\h(e_2)(a_3)=0,$$
			\[ [a_1,a_2]_\p=\al a_3,\;[a_1,a_3]_\p=pa_1+qa_2, [a_2,a_3]_\p=-qa_1+pa_2\esp \phi_\p(a_i)= \left( \begin{matrix}
			0&\mu_i\\0&0
			\end{matrix}  \right),i=1,2,3 \]
			and one of the following situations occurs:
			\begin{enumerate}
				\item[$(i)$]  $p\not=0$, $\al=0$ and 
				\[ \mu(a_1,a_2)=0,\;\mu(a_2,a_3)=-\la^{-1} (p\mu_1+q\mu_2)e_1\esp \mu(a_1,a_3)=\la^{-1}(-q\mu_1+p\mu_2)e_1. \]
				
				\item[$(ii)$] $p=0$,  $\mu_3\not=0$, $\al=0$ and
				\[\mu(a_1,a_2)=ce_1, \mu(a_2,a_3)=-\la^{-1} (p\mu_1+q\mu_2)e_1\esp \mu(a_1,a_3)=\la^{-1}(-q\mu_1+p\mu_2)e_1. \]
				\item[$(iii)$] $p=0$,  $\mu_3=0$ and 
				\[\mu(a_1,a_2)=c_1e_1+c_2e_2, \mu(a_2,a_3)=-\la^{-1} (p\mu_1+q\mu_2)e_1\esp \mu(a_1,a_3)=\la^{-1}(-q\mu_1+p\mu_2)e_1. \]
				 
			\end{enumerate}

		\end{prop}	\begin{proof}	
				Since $\phi_\h\not=0$ then $\phi_\h(\h)$ is a non trivial abelian subalgebra of $\mathrm{so}(\p)$ and hence it must be  one dimensional. Then there exists an orthonormal basis $(e_1,e_2)$ of $\h$ and an orthonormal basis $(a_1,a_2,a_3)$ of $\p$ and $\la>0$ such that $\phi_\h(e_1)=0$ and
				 $$\phi_\h(e_2)(a_1)=\la a_2,\; \phi_\h(e_2)(a_2)=-\la a_1\esp \phi_\h(e_2)(a_3)=0.$$
				 The first equation in \eqref{eqpro} is equivalent to 
				 \[ \phi_\p(\phi_\h(e_2)(a))(e_1)=0,\quad a\in\p. \]This is equivalent to
				 \[ \phi_\p(a_1)(e_1) =\phi_\p(a_2)(e_1)=0.\]
				Thus $\phi_\p(a_i)=\left( \begin{matrix}
				0&\mu_i\\0&0
				\end{matrix}  \right)$ for $i=1,2$ and $\phi_\p(a_3)=\left(\begin{matrix}
				u&v\\w&-u
				\end{matrix}  \right)$. Consider now the second equation in \eqref{eqpro}
				\[ \phi_\h(u)([a,b]_\p)=[a,\phi_\h(u)(b)]_\p+[\phi_\h(u)(a),b]_\p+\phi_\h(\phi_\p(b)(u))(a)-\phi_\h(\phi_\p(a)(u))(b). \]
				This equation is obviously true when $u=e_1$ and $(a,b)=(a_1,a_2)$. For $u=e_1$ and $(a,b)=(a_1,a_3)$, we get
				\[ \phi_\h(\phi_\p(a_3)(e_1))(a_1)=0 \]and hence $w=0$.

				 For $u=e_2$ and $(a,b)=(a_1,a_2)$, we get
				$ \phi_\h(e_2)([a_1,a_2]_\p)=0$ and hence
				$[a_1,a_2]_\p=\al a_3$.
				
				For $u=e_2$ and $(a,b)=(a_1,a_3)$ or $(a,b)=(a_2,a_3)$ , we get
				\[ \phi_\h(e_2)([a_1,a_3]_\p)=\la [a_2,a_3]_\p-\la u a_2 \esp \phi_\h(e_2)([a_2,a_3]_\p)=-\la [a_1,a_3]_\p+\la u a_1.\]
				This implies that $[a_1,a_3]_\p,[a_2,a_3]_\p\in\mathrm{span}\{ a_1,a_2 \}$ and hence
				\[ [a_1,a_3]_\p=pa_1+qa_2\esp [a_2,a_3]_\p=ra_1+sa_2. \]So
				\[ \begin{cases}
				\la(pa_2-qa_1)=\la(ra_1+sa_2-ua_2),\\
				\la(ra_2-sa_1)=-\la(pa_1+qa_2-ua_1).
				\end{cases} \]This is equivalent to
				\[ u=0, p=s\esp r=-q. \]
				To summarize, we get
				\[ [a_1,a_2]_\p=\al a_3,\;[a_1,a_3]_\p=pa_1+qa_2, [a_2,a_3]_\p=-qa_1+pa_2\esp \phi_\p(a_i)=\left( \begin{matrix}
				0&\mu_i\\0&0
				\end{matrix}  \right). \]
				
				Let consider now the fourth equation in \eqref{eqpro}
				\[ \phi_\p([a,b]_\p)(u)=[\phi_\p(a),\phi_\p(b)](u)+[u,{\mu}(a,b)]_\h-{\mu}(a,\phi_\h(u)(b))-{\mu}(\phi_\h(u)(a),b).
				 \]This equation is obviously true for $u=e_1$.
				  
				 For $u=e_2$ and $(a,b)=(a_1,a_2)$, $(a,b)=(a_1,a_3)$ or $(a,b)=(a_2,a_3)$, we get
				 \[ \begin{cases}
				 \al\mu_3=0,\\ (p\mu_1+q\mu_2)e_1=-\la\mu(a_2,a_3),\\
				 (-q\mu_1+p\mu_2)e_1=\la\mu(a_1,a_3).
				 \end{cases} \] The last two equations are equivalent to
				 \[ \phi_\p(a_3)(\mu(a_1,a_2))=-2p\mu(a_1,a_2)\esp p[a_1,a_2]_\p=0. \]
				 $\bullet$ $p\not=0$ then 
				 \[ \al=0,\mu(a_1,a_2)=0,\;\mu(a_2,a_3)=-\la^{-1} (p\mu_1+q\mu_2)e_1\esp \mu(a_1,a_3)=\la^{-1}(-q\mu_1+p\mu_2)e_1. \]
				 
				 $\bullet$ $p=0$ and $\mu_3\not=0$ then $\al=0$ and
				 \[\mu(a_1,a_2)=ce_1, \mu(a_2,a_3)=-\la^{-1} (p\mu_1+q\mu_2)e_1\esp \mu(a_1,a_3)=\la^{-1}(-q\mu_1+p\mu_2)e_1. \]
				 $\bullet$ $p=0$ and $\mu_3=0$ then 
				 \[\mu(a_1,a_2)=c_1e_1+c_2e_2, \mu(a_2,a_3)=-\la^{-1} (p\mu_1+q\mu_2)e_1\esp \mu(a_1,a_3)=\la^{-1}(-q\mu_1+p\mu_2)e_1. \]

			\end{proof}

			 By using Propositions \ref{p1}-\ref{ab3b}, we can give all the Riemann-Poisson Lie algebras of dimension 3, 4 or 5.

			 Let $(\G,\br,\va,r)$ be a Riemann-Poisson Lie algebra of dimension less or equal to 5. According to what above then $\G=\h\oplus\p$ and the Lie bracket on $\G$ is given by \eqref{bra} and $((\h,\om,\va_\h),(\p,\br_\p,\va_\p),\mu,\phi_\h,\phi_\p)$ are solutions  of Problem \ref{pro1}.
			 
			 $\bullet$ $\dim\G=3$. In this case $\dim\h=2$ and $\dim\p=1$ and, by applying Proposition \ref{p1}, the Lie bracket of $\G$, $\va$ and $r$ are given  in Table \ref{1}, where $e^{12}=e_1\wedge e_2$.

			 			 \begin{center}
			 	\begin{tabular}{|l|l|l|l|}
			 		\hline
			 		Non vanishing Lie brackets& Bivector $r$&Matrix of $\va$&Conditions\\
			 	\hline
			 $[e_1,e_2]=ae_1, [e_3,e_2]=b e_1$	&$\al e^{12}$&$\mathrm{I}_3$&$a\not=0,\al\not=0$\\
			 	\hline
			 $[e_3,e_1]=-b e_1+ce_2, [e_3,e_2]=d e_1+be_2$	&$\al e^{12}$&$\mathrm{I}_3$&$\al\not=0$\\
			 	\hline	
			 \end{tabular}
			 \captionof{table}{\label{1}Three dimensional Riemann-Poisson Lie algebras}
			 \end{center}
			 
			 $\bullet$ $\dim\G=4$. We have three cases:
			 \begin{enumerate}
			 	\item[$(c41)$] $\dim\h=2$, $\dim\p=2$ and $\h$ is non abelian and we can apply  Proposition \ref{nab} to get the Lie brackets on $\G$, $\va$ and $r$. They are described in rows 1 and 2 in Table \ref{2}.
			 	\item[$(c42)$] $\dim\h=2$, $\dim\p=2$ and $\h$ is  abelian and we can apply  Propositions \ref{ab} and \ref{sl} to get the Lie brackets on $\G$, $\va$ and $r$. They are described in rows 3 and 8 in Table \ref{2}.
			 	\item[$(c43)$] $\dim\h=4$. In this case $\G$ is a K\"ahler Lie algebra. We have used \cite{ovando} to derive all four dimensional K\"ahler Lie algebra together with their symplectic derivations. The results are given in Table \ref{3}. The notation $\mathrm{Der}^s(\h)$ stands for the  vector spaces of derivations which are skew-symmetric with respect the symplectic form. The vector space $\mathrm{Der}^s(\h)$ is described by a family of generators and $E_{ij}$ is the matrix with 1 in the $i$ row and $j$ column and 0 elsewhere. 
			 \end{enumerate}
			 
			  \begin{center}
			  	\begin{tabular}{|l|l|l|l|}
			  		\hline
			  		Non vanishing Lie brackets& Bivector $r$&Matrix of $\va$&Conditions\\
			  		\hline
			  		$[e_1,e_2]=ae_1,\;[e_3,e_2]=be_1+ce_4,$ 	&$\al e^{12}$&$\mathrm{I}_4$&$a\not=0,\al\not=0$\\
			  		$[e_4,e_2]=d e_1-ce_3$&&&\\
			  		\hline
			  		$[e_1,e_2]=ae_1,\;[e_3,e_2]=be_1,$	&$\al e^{12}$&$\mathrm{I}_4$&$\al ac\not=0,$\\
			  		 $[e_4,e_2]=d e_1,[e_3,e_4]=ce_3-a^{-1}cbe_1$&&&\\
			  		\hline
			  		$[e_3,e_4]=ae_1+be_2$	&$\al e^{12}$&$\mathrm{I}_4$&$\al\not=0$\\
			  		
			  		\hline
			  		$[e_3,e_4]=ae_1+be_2+ce_3,\;[e_4,e_1]=xe_1+ye_2, $	&$\al e^{12}$&$\mathrm{I}_4$&$\al\not=0$\\
			  		$[e_4,e_2]=ze_1-xe_2$&&&\\
			  		\hline
			  		$[e_3,e_4]=ae_1+be_2+2e_4,\;[e_3,e_1]=e_1,$	&$\al e^{12}$&$\mathrm{Diag}\left(1,1,\left(\begin{matrix}
			  		\mu&\nu\\\nu&\rho
			  		\end{matrix}\right)\right)$&$\al\not=0,\mu,\rho>0$\\
			  		$[e_3,e_2]=-e_2,[e_4,e_2]=e_1$&&&$\mu\rho>\nu^2$\\
			  		\hline
			  		$[e_3,e_4]=ae_1+be_2-2e_4,\;[e_3,e_1]=e_1,$	&$\al e^{12}$&$\mathrm{Diag}\left(1,1,\left(\begin{matrix}
			  		\mu&\nu\\\nu&\rho
			  		\end{matrix}\right)\right)$&$\al\not=0,\mu,\rho>0$\\
			  		$[e_3,e_2]=-e_2,[e_4,e_1]=e_2$&&&$\mu\rho>\nu^2$\\
			  		\hline
			  		$[e_3,e_4]=ae_1+be_2-2e_3,\;[e_3,e_1]=e_1+x e_2,$	&$\al e^{12}$&$\mathrm{Diag}\left(1,1,\left(\begin{matrix}
			  		\mu&\nu\\\nu&\rho
			  		\end{matrix}\right)\right)$&$\al\not=0,\mu,\rho>0$\\
			  		$[e_3,e_2]=-\frac1xe_1-e_2,[e_4,e_1]=xe_2,[e_4,e_2]=\frac1xe_1$&&&$\mu\rho>\nu^2,x\not=0$\\
			  		\hline
			  		$[e_3,e_4]=ae_1+be_2,\;[e_3,e_2]=xe_1+ye_4,$	&$\al e^{12}$&$\mathrm{I}_4$&$\al y\not=0$\\
			  		$[e_4,e_2]=ze_1-ye_3$&&&\\
			  		\hline	
			  	\end{tabular}
			  	\captionof{table}{\label{2}Four dimensional Riemann-Poisson Lie algebras of rank 2}
			  \end{center}		
			 
			 \begin{center}
			\begin{tabular}{|l|l|l|l|l|}
				\hline
				Non vanishing Lie brackets& Bivector $r$&Matrix of $\va$&Conditions&$\mathrm{Der}^s(\h)$\\
				\hline
				$[e_1,e_2]=e_2,$  &$\al e^{12}+\be e^{34}$ &$\mathrm{Diag}(a,b,c,d)$ &$\al\be\not=0$&$\{E_{21},E_{33}-E_{44},E_{43},E_{34}   \}$\\
				  & &  &$a,b,c,d>0$&\\
				\hline 
				$[e_1,e_2]=-e_3,[e_1,e_3]=e_2,$ &$\al e^{14}+\be e^{23}$ &$\mathrm{Diag}(a,b,b,c)$ &$\al\be\not=0$&$\{E_{23}-E_{32},E_{41}   \}$\\
				 & &  &$a,b,c>0$&\\
				\hline
				$[e_1,e_2]=e_2,[e_3,e_4]=e_4,$  &$\al e^{12}+\be e^{34}$ &$\mathrm{Diag}(a,b,c,d)$ &$\al\be\not=0$&$\{E_{21},E_{43} \}$\\
				  & &  &$a,b,c,d>0$&\\
				\hline
				$[e_4,e_1]=e_1,[e_4,e_2]=-\de e_3,$ &$\al e^{14}+\be e^{23}$ &$\mathrm{Diag}(a,b,b,c)$ &$\al\be\not=0,\de>0$&$\{E_{14},E_{23}-E_{32} \}$\\
				$[e_4,e_3]=\de e_2$ & &  &$a,b,c>0$&\\
				\hline
				$[e_1,e_2]=e_3,[e_4,e_3]= e_3,$ &$\al (e^{12}- e^{34})$ &$\mathrm{Diag}(a,\mu b,\mu a,b)$ &$\al\not=0$&$\{E_{34},E_{22}-E_{11},E_{12}+E_{21} \}$\\
				$[e_4,e_1]=\frac12 e_1,[e_4,e_2]=\frac12e_2,,$ & &  &$a,b,\mu>0$&\\
				\hline
				$[e_1,e_2]=e_3,[e_4,e_3]= e_3,$ &$\al (e^{23}+ e^{14})$ &$\mathrm{Diag}(a,a,2 a,2a)$ &$\al\not=0$&$\{2E_{14}-E_{32} \}$\\
				$[e_4,e_1]=2 e_1,[e_4,e_2]=-e_2,$ & &  &$a>0$&\\
				\hline
				$[e_1,e_2]=e_3,[e_4,e_3]= e_3,$ &$\al (e^{12}- e^{34})$ &$\mathrm{Diag}(a,a, a,a)$ &$\al\not=0$&$\{E_{34},E_{12}-E_{21} \}$\\
				$[e_4,e_1]=\frac12 e_1-e_2,$ & &  &$a>0$&\\
				$[e_4,e_2]=e_1+\frac12e_2,$&&&&\\
				\hline
			\end{tabular}		
			\captionof{table}{\label{3}Four-dimensional K\"ahler Lie algebras and their symplectic derivations}\end{center}
		$\bullet$ $\dim\G=5$. We have:
		\begin{enumerate}
			\item[$(c51)$] $\dim\h=4$ and $\h$ abelian and hence a symplectic vector space. We can apply Proposition \ref{p1} and $\G$ is semi-direct product.
			\item[$(c52)$] $\dim\h=4$ and $\h$ non abelian. We can apply Proposition \ref{p1} and Table \ref{3} to get the Lie brackets on $\G$, $\va$ and $r$. The result is summarized in Table \ref{4}.
			\item[$(c53)$] $\dim\h=2$ and $\h$ non abelian. We  apply Proposition \ref{nab}. In this case $(\p,\br_\p,\va_\p)$ is a 3-dimensional Euclidean Lie algebra and one must compute $\mathrm{Der}(\p)\cap\mathrm{so}(\p)$ and solve \eqref{eqh}. Three dimensional Euclidean Lie algebras were classified in \cite{lee}. For each of them we have computed $\mathrm{Der}(\p)\cap\mathrm{so}(\p)$ and solved \eqref{eqh} by using Maple. The result is summarized in Table \ref{5} when $\p$ is unimodular and Table \ref{6} when $\p$ is nonunimodular.
			\item[$(c54)$] $\dim\h=2$ and $\h$  abelian and $\phi_\h=0$. We  apply Proposition \ref{ab3} and we perform all the needed computations. We use the classification of 3-dimensional Euclidean Lie algebras given in \cite{lee}. The results are given in Tables \ref{7}-\ref{8}. 
			\item[$(c55)$] $\dim\h=2$ and $\h$  abelian and $\phi_\h\not=0$. We  apply Proposition \ref{ab3b} and we perform all the needed computations.  The results are given in Table \ref{9}.
		\end{enumerate}
		
		\begin{center}	
		\begin{tabular}{|l|l|l|l|}
		\hline
		Non vanishing Lie brackets& Bivector $r$&Matrix of $\va$&Conditions\\
		\hline
		$[e_1,e_2]=e_2,[e_5,e_1]=xe_2,$  &$\al e^{12}+\be e^{34}$ &$\mathrm{Diag}(a,b,c,d,e)$ &$\al\be\not=0$\\
		$[e_5,e_3]=ye_3+te_4,[e_5,e_4]=ze_3-ye_4$  & &  &$a,b,c,d,e>0$\\
		 \hline 
		$[e_1,e_2]=-e_3,[e_1,e_3]=e_2,$ &$\al e^{14}+\be e^{23}$ &$\mathrm{Diag}(a,b,b,c,d)$ &$\al\be\not=0$\\
		$[e_5,e_1]=ye_4,[e_5,e_2]=-xe_3,[e_5,e_3]=xe_2$ & &  &$a,b,c,d>0$\\
		 \hline
		 $[e_1,e_2]=e_2,[e_3,e_4]=e_4,$  &$\al e^{12}+\be e^{34}$ &$\mathrm{Diag}(a,b,c,d,e)$ &$\al\be\not=0$\\
		 $[e_5,e_1]=xe_2,[e_5,e_3]=ye_4$  & &  &$a,b,c,d,e>0$\\
		 \hline
		 $[e_4,e_1]=e_1,[e_4,e_2]=-\de e_3,[e_4,e_3]=\de e_2$ &$\al e^{14}+\be e^{23}$ &$\mathrm{Diag}(a,b,b,c,d)$ &$\al\be\not=0,\de>0$\\
		 $[e_5,e_2]=-ye_3,[e_5,e_3]=ye_2,[e_5,e_4]=xe_1$ & &  &$a,b,c,d>0$\\
		 \hline
		 $[e_1,e_2]=e_3,[e_4,e_3]= e_3,[e_4,e_1]=\frac12 e_1$ &$\al (e^{12}- e^{34})$ &$\mathrm{Diag}(a,\mu b,\mu a,b,c)$ &$\al\not=0$\\
		 $[e_4,e_2]=\frac12e_2,[e_5,e_1]=xe_1+ye_2,$ & &  &$a,b,c,\mu>0$\\
		 $[e_5,e_2]=ye_1-xe_2,[e_5,e_4]=ze_3$&&&\\
		 \hline
		 $[e_1,e_2]=e_3,[e_4,e_3]= e_3,[e_4,e_1]=2 e_1$ &$\al (e^{23}+ e^{14})$ &$\mathrm{Diag}(a,a,2 a,2a,b)$ &$\al\not=0$\\
		 $[e_4,e_2]=-e_2,[e_5,e_2]=xe_3,[e_5,e_4]=-2xe_1$ & &  &$a,b>0$\\
		 \hline
		 $[e_1,e_2]=e_3,[e_4,e_3]= e_3,[e_4,e_1]=\frac12 e_1-e_2$ &$\al (e^{12}- e^{34})$ &$\mathrm{Diag}(a,a, a,a,b)$ &$\al\not=0$\\
		 $[e_4,e_2]=e_1+\frac12e_2,[e_5,e_1]=-xe_2,[e_5,e_2]=xe_1$ & &  &$a,b>0$\\
		 $[e_5,e_4]=ye_3$&&&\\
		 \hline
		\end{tabular}		
		\captionof{table}{\label{4}Five-dimensional Riemann-Poisson Lie algebras of rank 4}\end{center}

		\begin{center}	
			\begin{tabular}{|l|l|l|l|}
				\hline
				Non vanishing Lie brackets& Bivector $r$&Matrix of $\va$&Conditions\\
				\hline
			$[e_1,e_2]= e_1,[e_3,e_2]=b\mu e_1-ce_4,[e_4,e_2]=d\mu e_1+ce_3$&$\al e^{12}$&$\mathrm{Diag}(1,\rho,\mu,\mu,1)$&$c\al\not=0$\\
			$[e_5,e_2]=fe_1,[e_3,e_4]=-fe_1+e_5$&&&$\mu,\rho>0$\\
			\hline
			$[e_1,e_2]= e_1,[e_3,e_2]=be_1,[e_4,e_2]=ce_1$&$\al e^{12}$&$\mathrm{Diag}(1,\rho,1,1,\mu)$&$\al\not=0$\\
			$[e_5,e_2]=d\mu e_1,[e_3,e_5]=b e_1-e_3,[e_4,e_5]=-ce_1+e_4$&&&$\mu,\rho>0$\\	
			\hline
			$[e_1,e_2]= e_1,[e_3,e_2]=(b+c)e_1,[e_4,e_2]=(cx+b)e_1$&$\al e^{12}$&$\mathrm{Diag}(1,\rho,\left(\begin{matrix}
			1&1\\1&x
			\end{matrix} \right),\mu)$&$\al\not=0$\\
			$[e_5,e_2]=d\mu e_1,[e_3,e_5]=(b+c) e_1-e_3,$&&&$\mu,\rho>0$\\
			$[e_4,e_5]=-(xc+b)e_1+e_4$&&&\\	
			\hline	
			$[e_1,e_2]= e_1,[e_3,e_2]=be_1,[e_4,e_2]=c\mu e_1$&$\al e^{12}$&$\mathrm{Diag}(1,\rho,1,\mu,\nu)$&$\al\not=0$\\
			$[e_5,e_2]=d\nu e_1,[e_3,e_5]=-\mu c e_1+e_4,[e_4,e_5]=be_1-e_3$&&&$\mu,\nu,\rho>0$\\	
			\hline	
			$[e_1,e_2]= e_1,[e_3,e_2]=b\mu e_1,[e_4,e_2]=c\nu e_1$&$\al e^{12}$&$\mathrm{Diag}(1,\xi,\mu,\nu,\rho)$&$\al\not=0$\\
			$[e_5,e_2]=d\rho e_1,[e_3,e_4]=-2\rho de_1+2e_5,$&&&$\mu,\nu,\rho,\xi>0$\\
			$[e_3,e_5]=2\nu c e_1-2e_4,[e_4,e_5]=2\mu be_1-2e_3$&&&$\mu\not=\nu,\mu\not=\rho,\nu\not=\rho$\\
			\hline	
			$[e_1,e_2]= e_1,[e_3,e_2]=b\mu e_1,[e_4,e_2]=c\nu e_1-\la e_5$&$\al e^{12}$&$\mathrm{Diag}(1,\rho,\mu,\nu,\nu)$&$\la\al\not=0$\\
			$[e_5,e_2]=d\nu e_1+\la e_4,[e_3,e_4]=-\frac{2\nu(\la c+d)}{1+\la^2} e_1+2e_5,$&&&$\mu,\nu,\rho>0$\\
			$[e_3,e_5]=\frac{2\nu(c-\la d)}{1+\la^2} e_1-2e_4,[e_4,e_5]=2\mu be_1-2e_3$&&&\\
			\hline
			$[e_1,e_2]= e_1,[e_3,e_2]=b\mu e_1,[e_4,e_2]=c\nu e_1$&$\al e^{12}$&$\mathrm{Diag}(1,\xi,\mu,\nu,\rho)$&$\al\not=0$\\
			$[e_5,e_2]=d\rho e_1,[e_3,e_4]=-\rho de_1+e_5,$&&&$\mu,\nu,\rho,\xi>0$\\
			$[e_3,e_5]=\nu c e_1-e_4,[e_4,e_5]=-\mu be_1+e_3$&&&$\mu\not=\nu,\mu\not=\rho,\nu\not=\rho$\\
			\hline
			$[e_1,e_2]= e_1,[e_3,e_2]=b\mu e_1,[e_4,e_2]=c\nu e_1-\la e_5$&$\al e^{12}$&$\mathrm{Diag}(1,\rho,\mu,\nu,\nu)$&$\la\al\not=0$\\
			$[e_5,e_2]=d\nu e_1+\la e_4,[e_3,e_4]=-\frac{\nu(\la c+d)}{1+\la^2} e_1+e_5,$&&&$\mu,\nu,\rho>0$\\
			$[e_3,e_5]=\frac{\nu(c-\la d)}{1+\la^2} e_1-e_4,[e_4,e_5]=-\mu be_1+e_3$&&&\\
			\hline	
			$[e_1,e_2]= e_1,[e_3,e_2]=b\mu e_1-ue_4-ve_5,$&$\al e^{12}$&$\mathrm{Diag}(1,\rho,\mu,\mu,\mu)$&$\al\not=0$\\
			$[e_4,e_2]=c\mu e_1+ue_3-w e_5,[e_5,e_2]=d\mu e_1+v e_3+we_4,$&&&$\mu,\rho>0$\\
			$[e_3,e_4]=x e_1+e_5,[e_3,e_5]=y e_1-e_4,[e_4,e_5]=ze_1+e_3$&&&\\
			$x=-\frac{\mu(buw-cuv+du^2+bv+cw+d)}{1+u^2+v^2+w^2}$&&&\\
			$y=\frac{\mu(-bvw+cv^2-duw+bu-dw+c)}{1+u^2+v^2+w^2}$&&&\\
			$z=-\frac{\mu(bw^2-cvw+duw-cu-dv+b)}{1+u^2+v^2+w^2}$&&&\\
			\hline	
	\end{tabular}		
	\captionof{table}{\label{5}Five-dimensional Riemann-Poisson Lie algebras of rank 2 with non abelian K\"ahler subalgebra and  unimodular complement}\end{center}

{\small
\begin{center}	
	\begin{tabular}{|l|l|l|l|}
		\hline
		Non vanishing Lie brackets& Bivector $r$&Matrix of $\va$&Conditions\\
		\hline
		$[e_1,e_2]= e_1,[e_3,e_2]=(f+c\la+f\la^2) e_1-\la e_4,$&$\al e^{12}$&$\mathrm{Diag}(1,\rho,1,1,\mu)$&$\la\al\not=0$\\
		$[e_4,e_2]=c e_1+\la e_3,[e_5,e_2]=d\mu e_1,[e_3,e_5]=fe_1-e_3,$&&&$\mu,\rho>0$\\
		$[e_4,e_5]=(\la f+c)e_1-e_4$&&&\\
		\hline
		$[e_1,e_2]= e_1,[e_3,e_2]=b e_1,[e_4,e_2]=c\mu e_1,$&$\al e^{12}$&$\mathrm{Diag}(1,\rho,1,\mu,\nu)$&$\al\not=0, f=1\;$or \\
		$[e_5,e_2]=d\nu e_1,[e_3,e_5]=\mu ce_1-e_4,$&&&$f\leq0,0<\mu<|f|,\rho>0$\\
		$[e_4,e_5]=(-fb+2\mu c)e_1+fe_3-2e_4$&&&\\
		\hline
		$[e_1,e_2]= e_1,[e_3,e_2]=(b+c\mu) e_1,[e_4,e_2]=(c+b\mu) e_1,$&$\al e^{12}$&$\mathrm{Diag}(1,\rho,\left(\begin{matrix}
		1&\mu\\\mu&1
		\end{matrix}  \right),\nu)$&$\al\not=0, $ \\
		$[e_5,e_2]=d\nu e_1,[e_3,e_5]=(\mu b+c) e_1-e_4,$&&&$\mu,\nu,\rho>0$\\
		$[e_4,e_5]=((2-\mu)c+(2\mu-1)b)e_1+e_3-2e_4$&&&\\
		\hline
		$[e_1,e_2]= e_1,[e_3,e_2]=(b+c) e_1,[e_4,e_2]=(b+c\mu) e_1,$&$\al e^{12}$&$\mathrm{Diag}(1,\rho,\left(\begin{matrix}
		1&1\\1&\mu
		\end{matrix}  \right),\nu)$&$\al\not=0, $ \\
		$[e_5,e_2]=d\nu e_1,[e_3,e_5]=( b+c\mu) e_1-e_4,$&&&$\nu,\rho>0,c>\mu>1$\\
		$[e_4,e_5]=((2-f)b+(2\mu-f)c)e_1+fe_3-2e_4$&&&\\
		\hline
		$[e_1,e_2]= e_1,[e_3,e_2]=(b+\frac12c) e_1,[e_4,e_2]=(c+\frac12b) e_1,$&$\al e^{12}$&$\mathrm{Diag}(1,\rho,\left(\begin{matrix}
		1&\frac12\\\frac12&1
		\end{matrix}  \right),\nu)$&$\al\not=0, $ \\
		$[e_5,e_2]=d\nu e_1,[e_3,e_5]=(c+\frac12b)e_1-e_4,$&&&$\rho,\nu>0$\\
		$[e_4,e_5]=(b+2 c)e_1-2e_4$&&&\\
		\hline
		$[e_1,e_2]= e_1,[e_3,e_2]=x e_1,[e_4,e_2]=y e_1,$&$\al e^{12}$&$A^tBA$&$\al\not=0, $ \\
		$[e_5,e_2]=d\nu e_1,[e_3,e_5]=ze_1-e_4,$&&$A=\left(  \begin{matrix}
		\frac{1+s}{-2fs}&-\frac{1}{2s}&0\\\frac{1-s}{2fs}&\frac{1}{2s}\\0&0&1
		\end{matrix}\right)$&$0<f<1,0\leq\mu<1,\nu,\rho>0$\\
		$[e_4,e_5]=te_1+fe_3-2e_4$&&$B=\mathrm{Diag}(1,\rho,\left(\begin{matrix}
		1&\mu\\\mu&1
		\end{matrix}  \right),\nu)$&\\
		$x=\frac{((\mu+1)b+(\mu-1)c)f-2b}{2f^2(f-1)},y=z=\frac{(\mu-1)(cf+b)}{2f(f-1)}$&&&\\
		$t=\frac{(1-\mu)cf+((f-2)\mu+f)b}{2f(1-f)}$&&$s=\sqrt{1-f}$&\\
		\hline
\end{tabular}		
\captionof{table}{\label{6}Five-dimensional Riemann-Poisson Lie algebras of rank 2 with non abelian K\"ahler subalgebra and non unimodular complement}\end{center}}

\begin{center}	
	\begin{tabular}{|l|l|l|l|}
		\hline
		Non vanishing Lie brackets& Bivector $r$&Matrix of $\va$&Conditions\\
		\hline
		$[e_3,e_4]=ae_1+be_2+e_5,[e_3,e_5]=ce_1+de_2$&$\al e^{12}$&$\mathrm{Diag}(1,1,\mu,\mu,1)$&$\al\not=0$\\
		$[e_4,e_5]=fe_1+ge_2$&&&$\mu>0$\\
		\hline
	$[e_3,e_4]=ae_1+be_2,[e_3,e_5]=ce_1+de_2-e_3$&$\al e^{12}$&$\mathrm{Diag}(1,1,1,1,\mu)$&$\al\not=0$\\
	$[e_4,e_5]=fe_1+ge_2+e_4$&&$\mathrm{Diag}(1,1,\left(\begin{matrix}
	1&1\\1&x
	\end{matrix} \right),\mu)$&$\mu>0$\\
	\hline
	$[e_3,e_4]=ae_1+be_2,[e_3,e_5]=ce_1+de_2+e_4$&$\al e^{12}$&$\mathrm{Diag}(1,1,1,\mu,\nu)$&$\al\not=0$\\
	$[e_4,e_5]=fe_1+ge_2-e_3$&&&$\mu,\nu>0$\\
	\hline
	$[e_3,e_4]=ae_1+be_2+2e_5,[e_3,e_5]=ce_1+de_2-2e_4$&$\al e^{12}$&$\mathrm{Diag}(1,1,\mu,\nu,\rho)$&$\al\not=0$\\
	$[e_4,e_5]=fe_1+ge_2-2e_3$&&&$\mu,\nu,\rho>0$\\
	\hline
	$[e_3,e_4]=ae_1+be_2+e_5,[e_3,e_5]=ce_1+de_2-e_4$&$\al e^{12}$&$\mathrm{Diag}(1,1,\mu,\nu,\rho)$&$\al\not=0$\\
	$[e_4,e_5]=fe_1+ge_2+e_3$&&&$\mu,\nu,\rho>0$\\
	\hline
	$[e_3,e_5]=ce_1+de_2-e_3$&$\al e^{12}$&$\mathrm{Diag}(1,1,1,1,\mu)$&$\al\not=0$\\
	$[e_4,e_5]=fe_1+ge_2-e_4$&&&$\mu>0$\\
	\hline
	$[e_3,e_5]=ce_1+de_2-e_4$&$\al e^{12}$&{ There are many cases}&$\al\not=0$\\
	$[e_4,e_5]=fe_1+ge_2+xe_3-2e_4$&&See \cite{lee}&\\
	\hline
\end{tabular}		
\captionof{table}{\label{7}Five-dimensional Riemann-Poisson Lie algebras of rank 2 with  abelian K\"ahler subalgebra}\end{center}
	{\small	
\begin{center}	
	\begin{tabular}{|l|l|l|l|}
		\hline
		Non vanishing Lie brackets& Bivector $r$&Matrix of $\va$&Conditions\\
		\hline
		$[e_3,e_1]=-e_2,[e_3,e_2]=e_1,[e_4,e_1]=e_2,[e_4,e_2]=e_1$&$\al e^{12}$&$\mathrm{Diag}(1,1,\mu,\nu,\rho)$&$\al\not=0$\\
		$[e_5,e_1]=e_1,[e_5,e_2]=-e_2,$&&&$\mu,\nu,\rho>0$\\
		$[e_3,e_4]=2e_5+(l_{22}-l_{21}-2l_{13})e_1-(l_{12}+l_{11}+2l_{23})e_2$&&&\\
		$[e_3,e_5]=-2e_4+(l_{23}-l_{11}+2l_{12})e_1-(l_{13}-l_{21}-2l_{22})e_2,$&&&\\
		$[e_4,e_5]=-2e_3+(l_{23}-l_{12}+2l_{11})e_1+(l_{13}+l_{22}+2l_{21})e_2$&&&\\
		\hline	
		$[e_4,e_2]=ue_1,[e_5,e_1]=-\frac{a}2e_1,[e_5,e_2]=ve_1+\frac{a}2e_2,$&$\al e^{12}$&$\mathrm{Diag}(1,1,1,1,1)$&$\al\not=0$\\
		$[e_3,e_4]=xe_1+ye_2,[e_3,e_5]=be_3+ze_1+te_2,$&&&$a\not=0,b\not=0$\\
		$[e_4,e_5]=ce_3+ae_4+re_1+se_2,$&&&$(3a+2b)y=0$\\
		&&&$(a+2b)x-2tu+2yv=0$\\
		\hline
		$[e_4,e_2]=ue_1,[e_5,e_1]=-\frac{a}2e_1,[e_5,e_2]=ve_1+\frac{a}2e_2,$&$\al e^{12}$&$\mathrm{Diag}(1,1,\left(\begin{matrix}
		1&\mu\\\mu&1
		\end{matrix}
		\right),1)$&$\al\not=0$\\
		$[e_3,e_4]=xe_1,[e_3,e_5]=ze_1+te_2,$&&&$a\not=0,$\\
		$[e_4,e_5]=ae_4+re_1+se_2,$&&&$ax-2tu=0$\\
		\hline
		$[e_4,e_1]=ue_2,[e_5,e_1]=\frac{a}2e_1+ve_2,[e_5,e_2]=-\frac{a}2e_2,$&$\al e^{12}$&$\mathrm{Diag}(1,1,1,1,1)$&$\al\not=0$\\
		$[e_3,e_4]=xe_1+ye_2,[e_3,e_5]=be_3+ze_1+te_2,$&&&$a\not=0,b\not=0$\\
		$[e_4,e_5]=ce_3+ae_4+re_1+se_2,$&&&$(3a+2b)x=0$\\
		&&&$(a+2b)y-2zu+2xv=0$\\
		\hline
		$[e_4,e_1]=ue_2,[e_5,e_1]=\frac{a}2e_1+ve_2,[e_5,e_2]=-\frac{a}2e_2,$&$\al e^{12}$&$\mathrm{Diag}(1,1,\left(\begin{matrix}
		1&\mu\\\mu&1
		\end{matrix}
		\right),1)$&$\al\not=0$\\
		$[e_3,e_4]=ye_2,[e_3,e_5]=ze_1+te_2,$&&&$a\not=0,$\\
		$[e_4,e_5]=ae_4+re_1+se_2,$&&&$ay-2zu=0$\\
		\hline
		$[e_4,e_1]=ue_1+upe_2,[e_4,e_2]=-\frac{u}pe_1-ue_2,,$&$\al e^{12}$&$\mathrm{Diag}(1,1,1,1,1)$&$\al\not=0$\\
		$[e_5,e_1]=ve_1+\frac{(2v-a)p}2e_2,[e_5,e_2]=-\frac{(2v+a)}{2p}e_1-ve_2$&&&\\
		$[e_3,e_4]=xe_1+ye_2,[e_3,e_5]=be_3+ze_1+te_2,$&&&$a\not=0,b\not=0$\\
		$[e_4,e_5]=ce_3+ae_4+re_1+se_2,$&&&\\
		&&&\\$((2a+2b+2v)x-2zu)p-ay+2tu-2yv=0$&&&\\
		$(2xv-ax-2zu)p+(2a+2b-2v)y+2tu=0$&&&\\
		\hline
		$[e_4,e_1]=ue_1+upe_2,[e_4,e_2]=-\frac{u}pe_1-ue_2,,$&$\al e^{12}$&$\mathrm{Diag}(1,1,\left(\begin{matrix}
		1&\mu\\\mu&1
		\end{matrix}
		\right),1)$&$\al\not=0$\\
		$[e_5,e_1]=ve_1+\frac{(2v-a)p}2e_2,[e_5,e_2]=-\frac{(2v+a)}{2p}e_1-ve_2$&&&\\
		$[e_3,e_4]=xe_1+ye_2,[e_3,e_5]=ze_1+te_2,$&&&$a\not=0,b\not=0$\\
		$[e_4,e_5]=ae_4+re_1+se_2,$&&&\\
		&&&\\$((2a+2v)x-2zu)p-ay+2tu-2yv=0$&&&\\
		$(2xv-ax-2zu)p+(2a-2v)y+2tu=0$&&&\\
		\hline
		$[e_5,e_1]=ue_1+ve_2,[e_5,e_2]=we_1-ue_2,$&$\al e^{12}$&$\mathrm{Diag}(1,1,1,1,1)$&$\al\not=0$\\
		$[e_3,e_4]=xe_1+ye_2,[e_3,e_5]=ae_3+be_4+ze_1+te_2,$&&&$(a+d+u)x+yw=0$\\
		$[e_4,e_5]=ce_3+de_4+re_1+se_2,$&&&$xv+(a+d-u)y=0$\\
		\hline
		$[e_5,e_1]=ue_1+ve_2,[e_5,e_2]=we_1-ue_2,$&$\al e^{12}$&$\mathrm{Diag}(1,1,1,1,1)$&$\al\not=0$\\
		$[e_3,e_4]=xe_1+ye_2+ae_4,[e_3,e_5]=be_4+ze_1+te_2,$&&&$a\not=0$\\
		$[e_4,e_5]=ce_4+re_1+se_2,$&&&$(c+u)x-ar+yw=0$\\
		&&&$(c-u)y-as+xv=0$\\
		\hline
\end{tabular}		
\captionof{table}{\label{8}Five-dimensional Riemann-Poisson Lie algebras of rank 2 with  abelian K\"ahler subalgebra (Continued)}\end{center}}

\begin{center}	
	\begin{tabular}{|l|l|l|l|}
		\hline
		Non vanishing Lie brackets& Bivector $r$&Matrix of $\va$&Conditions\\
		\hline
		$[e_3,e_2]=xe_1-ae_4,[e_4,e_2]=ye_1+ae_3,[e_5,e_2]=ze_1$&$\al e^{12}$&$\mathrm{Diag}(1,1,1,1,1)$&$\al\not=0$\\
		$[e_3,e_5]=pe_3+qe_4+a^{-1}(-qx+py)e_1,$&&&$a\not=0$\\
		$[e_3,e_5]=-qe_3+pe_4-a^{-1}(px+qy)e_1$&&&\\
		\hline
		$[e_3,e_2]=xe_1-ae_4,[e_4,e_2]=ye_1+ae_3,[e_5,e_2]=ze_1$&$\al e^{12}$&$\mathrm{Diag}(1,1,1,1,1)$&$\al\not=0$\\
		$[e_3,e_4]=be_1$&&&\\
		$[e_3,e_5]=qe_4-a^{-1}qxe_1,$&&&$a\not=0,z\not=0$\\
		$[e_3,e_5]=-qe_3-a^{-1}qye_1$&&&\\
		\hline
		$[e_3,e_2]=xe_1-ae_4,[e_4,e_2]=ye_1+ae_3,$&$\al e^{12}$&$\mathrm{Diag}(1,1,1,1,1)$&$\al\not=0$\\
		$[e_3,e_4]=be_1+ce_2$&&&\\
		$[e_3,e_5]=qe_4-a^{-1}qxe_1,$&&&$a\not=0$\\
		$[e_3,e_5]=-qe_3-a^{-1}qye_1$&&&\\
		\hline				
\end{tabular}		
\captionof{table}{\label{9}Five-dimensional Riemann-Poisson Lie algebras of rank 2 with  abelian K\"ahler subalgebra (Continued)}\end{center}
	
This theorem unknown to our knowledge can be used to build examples of Riemann-Poisson Lie algebras.

\begin{thm} Let $(G,\prs)$ be an even dimensional flat Riemannian Lie group. Then there exists a left invariant differential $\Om$ on $G$ such that $(G,\prs,\Om)$ is a K\"ahler Lie group.
	
\end{thm}

\begin{proof} Let $\G$ be the Lie algebra of $G$ and $\va=\prs(e)$. According to Milnor's Theorem \cite[Theorem 1.5]{milnor} and its improved version \cite[Theorem 3.1]{ait} the flatness of the metric on $G$ is equivalent to $[\G,\G]$ is even dimensional abelian, $[\G,\G]^\perp=\{u\in\G,\ad_u+\ad_u^*=0 \}$ is also even dimensional abelian and $\G=[\G,\G]\oplus [\G,\G]^\perp$. Moreover, the Levi-Civita product is given by\begin{equation}\label{eq12}\mathrm{L}_a=\left\{\begin{array}{ccc}
	\ad_a&\mbox{if}
	&a\in [\G,\G]^\perp,
	\\ 0&\mbox{if}&a\in [\G,\G]
	\end{array}\right.\end{equation} and there exists a basis $(e_1,f_1,\ldots,e_r,f_r)$ of $[\G,\G]$ and $\la_1,\ldots,\la_r\in[\G,\G]^\perp\setminus\{0\}$ such that for any $a\in[\G,\G]^\perp$,
	\[ [a,e_i]=\la_i(a)f_i\esp [a,f_i]=-\la_i(a)e_i. \]
	We consider a nondegenerate skew-symmetric 2-form $\om_0$ on $[\G,\G]^\perp$ and $\om_1$ the nondegenerate skew-symmetric 2-form on $[\G,\G]^\perp$ given by $\om_1=\sum_{i=1}^r e_i^*\wedge f_i^*$. One can sees easily that $\om=\om_0\oplus\om_1$ is a K\"ahler form on $\G$.
	\end{proof}

\end{document}